\documentclass[a4,11pt]{article}
	\usepackage{amsthm}
	\usepackage{amsmath}
	\usepackage{amssymb}
	\usepackage{mathrsfs}
	\usepackage[T1]{fontenc}
	\usepackage{enumerate}
	\usepackage{bbm}
	\usepackage{color}
	
\renewcommand{\mathbb}{\mathbbm}                     
	
\renewcommand{\epsilon}{\varepsilon}                 
\renewcommand{\phi}{\varphi}
\renewcommand{\le}{\leqslant}
\renewcommand{\ge}{\geqslant}
\renewcommand{\leq}{\le}
\renewcommand{\geq}{\ge}

\DeclareMathOperator{\C}{{\mathbb  C}}                
\DeclareMathOperator{\R}{{\mathbb R}}                
\DeclareMathOperator{\N}{{\mathbb N}}                
\DeclareMathOperator{\Borel}{{\mathfrak B}}

\newcommand{\scapro}[2]{\langle #1,#2\rangle}       

\allowdisplaybreaks                                                                                                
\numberwithin{equation}{section}

\theoremstyle{plain}
\newtheorem{thm}{\protect\theoremname}[section]		
\theoremstyle{definition}

\theoremstyle{plain}
\theoremstyle{definition}
\newtheorem{example}[thm]{Example}
\newtheorem{lemma}[thm]{Lemma}
\newtheorem{remark}[thm]{Remark}

 \providecommand{\definitionname}{Definition}
\providecommand{\theoremname}{Theorem}	
	
\newcommand{\ud}{\,\mathrm{d}}
\newcommand{\udd}{\mathrm{d}}

\newcommand{\1}{\mathbbm{1}}

\newcommand\numberthis{\addtocounter{equation}{1}\tag{\theequation}}
	
\newcommand{\itemEq}[1]{%
	\begingroup%
   	\setlength{\abovedisplayskip}{0pt}%
	\setlength{\belowdisplayskip}{0pt}%
	\parbox[c]{\linewidth}{\begin{flalign}#1&&\end{flalign}}%
	\endgroup}
	
	\newcommand{\beq}{\begin{equation}}
	\newcommand{\beql}[1]{\begin{equation}\label{#1}}
	\newcommand{\eeq}{\end{equation}}

\title{Invariant measure for the stochastic Cauchy problem driven by a cylindrical L\'evy process}

\author{Umesh Kumar \\Department of Mathematics\\
Rajdhani College (University of Delhi)\\
New Delhi  110015 \\ India\\ \\
umesh.kumar@rajdhani.du.ac.in
 \and Markus Riedle \\Department of Mathematics \\ King's College  \\ London WC2R 2LS\\ United Kingdom\\ \\ markus.riedle@kcl.ac.uk}

\date{5 April 2019}

\begin{document}
\maketitle
\begin{abstract}
In this work, we present sufficient conditions for the existence of a stationary solution of  an abstract stochastic Cauchy problem driven by an arbitrary cylindrical L\'evy process,  and show that these conditions are also necessary if the semigroup is stable, in which case the invariant measure is unique. For typical situations such as the heat equation, we significantly simplify these conditions without assuming any further restrictions on the driving cylindrical L\'evy process and demonstrate their application in some examples.

	\end{abstract}
\section{Introduction}
Cylindrical L\'evy processes naturally extend the class of cylindrical Brownian motions and cover many examples of L\'evy-type noise considered in the literature. A general framework of cylindrical L\'evy
processes in Banach spaces has been recently introduced by Applebaum and Riedle in \cite{app}. Stochastic integration of deterministic operator-valued integrands with respect to cylindrical L\'evy processes is developed in \cite{OU}. Based on this integration theory, the authors of the present article have developed a general theory of weak and  mild solutions for  the stochastic Cauchy problem driven by an arbitrary cylindrical L\'evy process in \cite{kumar_riedle}. 

More specifically, the stochastic Cauchy problem is a linear evolution equation driven by an additive noise of the form
\begin{equation}\label{SCP10}
\ud Y(t)= AY(t)\ud t+B\ud L(t)  \qquad  \text{for all $t \in [0,T]$.} 
\end{equation}
Here, $L$ is a cylindrical L\'evy process on a separable Hilbert space $U$,
the coefficient $A$ is the generator of a strongly continuous semigroup $(T(t))_{t\ge 0}$ on a separable Hilbert space $V$ and $B$  is a linear, bounded operator from $U$ to $V$. In this general setting, we 
present sufficient conditions for the existence of a stationary solution of \eqref{SCP10} and show that these conditions are also necessary if the semigroup is stable, in which case the invariant measure is unique. If the semigroup has a spectral decomposition, we significantly simplify these conditions without assuming any further restrictions on the driving cylindrical L\'evy process. 

For finite dimensional L\'evy processes the existence of invariant measures and its relation to operator self-decomposibilty has been studied by  Jurek \cite{jurek},  Jurek and Vervaat  \cite{jurek_vervaat},  Sato and Yamatado \cite{sato_yamazato_1}, \cite{sato_yamazato_2},  Wolfe \cite{wolfe}, and Zabczyk \cite{zabczyk_stationary}.
The case of an infinite dimensional L\'evy process in a Hilbert space was studied by Chojnowska-Michalik in \cite{michalik} and \cite{michalik_heat}. To the best of our knowledge, the case of a cylindrical L\'evy process was only considered for a specific example of a cylindrical L\'evy process
and under further assumptions on the semigroup in  \cite{priola_zabczyk_cyl}. The assumptions in \cite{priola_zabczyk_cyl} enable the authors to reduce the problem of the existence of an invariant measure to the analogue problem in one dimension. The general setting in the present paper clearly excludes this approach. Our results in the general framework can easily be applied to the example considered in  \cite{priola_zabczyk_cyl}, and we are not only able to cover these results but even improve them; see Example \ref{ex.priola}.

In our general framework, having in hand the integration theory developed in \cite{OU} and the probabilistic description of cylindrical L\'evy processes by their characteristics introduced in \cite{Riedle_infinitely}, we are able to generalise the conditions from the case of a genuine L\'evy process in \cite{michalik} to the cylindrical setting. The fact, that cylindrical processes are generalised processes not attaining values in the underlying Hilbert space, prevents us from directly adopting the methods from the classical case. Instead, we exploit some of the methods developed in \cite{kumar_riedle} and \cite{OU} such as tightness of finite-dimensional approximations. As in the classical setting, the derived conditions are rather difficult to verify in the general case but can be significantly simplified in typical cases such as the heat equation; see \cite{michalik_heat} for the classical case. Again, the fact that cylindrical processes are generalised processes not attaining values in the underlying Hilbert space requires some more advanced arguments. 

Our article begins with Section \ref{se.preliniaries} where we fix most of our notations and introduce cylindrical L\'evy processes and their stochastic integral. In section \ref{se.invariant} we briefly demonstrate the equivalence of 
the existence of a stationary solution of \eqref{SCP10} and of an invariant measure for the corresponding Mehler semigroup. Our first main result of this article is presented in this section, which provides sufficient conditions for the existence of a stationary solution in terms of the characteristics of the driving cylindrical L\'evy process. Our second main result is in the final Section \ref{se.stable_case}, where  we significantly simplify the conditions from Section \ref{se.invariant} if the semigroup has a spectral decomposition. We finish the article by demonstrating our results in some examples.

\section{Preliminaries}\label{se.preliniaries}
Let $U$  and $V$ be  real separable Hilbert spaces with norms $\|\cdot\|$ and inner products $\scapro{\cdot}{\cdot}$. Let  $(e_k)_{k\in \mathbb{N}}$ and $(h_k)_{k\in \mathbb{N}}$ be the orthonormal bases of $U$ and $V$, respectively. We identify the dual of a Hilbert space by the space itself. The space of all linear, bounded operators from $U$ to $V$ is 
denoted by $\mathcal{L}(U,V)$, equipped with the operator norm $\|\cdot \|_{\text{op}}$. By $B_U$, we denote the open unit ball in $U$, that is, $B_U:=\{u\in U: \|u\| < 1\}$. The Borel $\sigma$-algebra of $U$ is denoted by $\Borel(U)$ and the
 space of Radon probability measures on $\Borel(U)$ is denoted by 
${\mathcal M}(U)$ and is equipped with the Prokhorov metric.

  We fix a filtered probability space $(\Omega, \mathcal{F}, \{\mathcal{F}_t\}_{t\ge 0}, P)$, where the filtration $\{\mathcal{F}_t\}_{t\ge 0}$ satisfies the usual conditions of right continuity and completeness.
 By $L^0_P(\Omega; U)$,  we denote 
the space of all equivalence classes of measurable functions $ g\colon \Omega \rightarrow U$ and it is equipped with the topology of convergence in probability.  The space of all regulated functions $g \colon [0, T] \to U$
is denoted by $R([0,T];U)$ and it is a Banach space when equipped 
with the supremum norm. Recall that a function $g \colon [0, T] \to U$ is called \emph{regulated} if it can be uniformly approximated by step functions. In particular, a regulated function has only countable number of discontinuities; see \cite[Ch.II.1.3]{bourbaki} for this and other properties 
we will use. 

Let $\Gamma$ be a subset of $U$. The sets of the form 
\[ C(u_1, ... , u_n; B) :=\{ u \in U: (\langle u, u_1 \rangle, ... , \langle u, u_n \rangle) \in B\},\]
for $u_1, ... , u_n \in \Gamma$ and $B\in \Borel(\mathbb{R}^n)$ are called {\em cylindrical sets with respect to $\Gamma$}. The set of all these cylindrical sets is denoted by $\mathcal{Z}(U,\Gamma)$ and it is a $\sigma$-algebra if $\Gamma$ is finite and otherwise an algebra. We write $\mathcal{Z}(U)$ for $\mathcal{Z}(U,U)$.  A function $\mu \colon \mathcal{Z}(U) \to [0,1]$ is called a \emph{cylindrical measure}, if for each finite subset $\Gamma \subseteq U$ the restriction of $\mu$ on the $\sigma$-algebra $\mathcal{Z}(U, \Gamma)$ is a                                                                                                                                                                                                                                                                                                                                                                                                                                                                                                                                                                                                                                    measure. A cylindrical measure $\mu$ is only finitely additive and is said to extend to a measure $\nu$ on $\Borel (U)$ if $\mu = \nu$ on $\mathcal{Z}(U)$.  A cylindrical measure is called finite if $\mu (U) < \infty$ and a cylindrical probability measure if $\mu(U) =1$.
A \emph{cylindrical random variable} $Z$ in $U$ is defined as a linear and continuous map $ Z\colon U \rightarrow L_P^0(\Omega; \mathbb{R})$.
Given a cylindrical random variable $Z$, we can  define a cylindrical probability measure $\lambda$ by 
\begin{align*}
\lambda\colon  \mathcal{Z}(U) \to [0,1],\qquad
\lambda(Z)=P\big( (Zu_1,\dots, Zu_n)\in B\big)
\end{align*}
for cylindrical sets $Z=C(u_1, ... , u_n; B)$. The cylindrical probability measure $\lambda$ is called the {\em cylindrical distribution} of $Z$. The characteristic function of a cylindrical random variable $Z$ is defined by
\[\phi_{Z}\colon U \rightarrow \mathbb{C}, \qquad \phi_{Z}(u)=E[\exp (iZu)],\]
and it uniquely determines the cylindrical distribution of $Z$.

A family $(Z(t):\, t\ge 0)$ of cylindrical random variables is called 
a {\em cylindrical process}. By a \emph{cylindrical L\'evy process}  we mean a cylindrical process $(L(t):\, t\ge 0)$ such that for all $u_1, ... , u_n \in U$ and $n\in \mathbb{N}$, the stochastic process
$((L(t)u_1, ... , L(t)u_n): t \ge 0)$
is a L\'evy process in $\mathbb{R}^n$ with respect to the filtration $\{\mathcal{F}_t\}_{t\ge 0}$. The characteristic function of $L(t)$ for all $t \ge 0$ is given by
\[\phi_{L(t)} \colon U \rightarrow \mathbb{C}, \qquad \phi_{L(t)}(u)=\exp\big(t\Psi(u)\big),\]
where $\Psi \colon U \rightarrow \mathbb{C}$ is called the (cylindrical) symbol of $L$, and is given by
\begin{align}\label{cyl_levy_symbol}
\Psi(u) = ia(u) - \frac{1}{2}\langle Qu, u\rangle +\int_U\left(e^{i\langle u, h \rangle}-1-i\langle u, h \rangle \1_{B_{\mathbb{R}}}(\langle u, h \rangle)\right)\mu (\udd h),
\end{align}
where $a \colon U \rightarrow \mathbb{R}$ is a continuous mapping with $a(0)=0$,
the mapping $Q \colon U \rightarrow U$ is a positive, symmetric operator  and $\mu$ is a cylindrical L\'evy measure on $\mathcal{Z}(U)$, that is it is a cylindrical measure on $\mathcal{Z}(U)$ satisfying
\[\int_U \left( \langle u, h \rangle^2 \wedge 1 \right) \mu(\udd h) < \infty \qquad \mathrm{for\;all\;}u \in U.\]
We call $(a, Q, \mu)$ the \emph{(cylindrical) characteristics of $L$}. Cylindrical L\'evy processes are introduced in \cite{app} and its characteristics further studied in \cite{Riedle_infinitely}. 

For a function $f\colon [0,T] \to \mathcal{L}(U,V)$ such that the map $f^*(\cdot)v\colon [0,T]\to U$ is a regulated function for each $v\in V$, one can define
the stochastic integral
\[Z_A(v):=\int_0^T\1_A(s)f^*(s)v \ud L(s)\]
for each set $A\in\Borel([0,T])$. In this way, one obtains a cylindrical random variable $Z_A \colon V \to L^0_P(\Omega;\R)$. The function $f$ is called \emph{stochastically integrable} with respect to $L$ if for each Borel set $A \in \Borel([0,T])$, the cylindrical random variable $Z_A$ extends to a genuine $V$-valued random variable $I_A$, that is 
\begin{align}
\scapro{I_A}{v}=\int_0^T\1_A(s)f^*(s)v \ud L(s)
\qquad \text{for all }v\in V. \label{eq.def.stoch_int}
\end{align}
This stochastic integration theory is developed in \cite{OU} and applied in  \cite{kumar_riedle} to study the weak solution of abstract stochastic Cauchy problem driven by a cylindrical L\'evy process. 

\section{Invariant measure}\label{se.invariant}
The main aim of this paper is to study the conditions for the existence of an invariant measure  for the solution  of the stochastic Cauchy problem
\begin{equation}\label{SCP1}
\begin{split}
\ud Y(t)&= AY(t)\ud t+B\ud L(t)  \qquad  \mathrm{for\; all}\;  t\ge 0, \\
Y(0)& =  Y_0,
\end{split}
\end{equation}
where $A$ is the generator of a $C_0$-semigroup $(T(t))_{t\geq 0}$
on a separable Hilbert space $V$, the driving noise  $L$ is a cylindrical L\'evy process on a separable Hilbert space $U$ and $B \colon U \rightarrow V$  is a bounded linear operator from $U$ to $V$. The initial condition $Y_0$ is a $V$-valued $\mathcal{F}_0$-measurable random variable. 

A $V$-valued process $(Y(t):t\in[0,T])$ is called a weak solution of \eqref{SCP1} on $[0,T]$ if it satisfies the following:
\begin{enumerate}[(1)]
	\item $Y$ is progressively measurable;
	\item  the mapping
	$t\mapsto \scapro{Y(t)}{g(t)}$ is integrable on $[0,T]$ for each $g\in C([0,T];V)$ and satisfies for each sequence $(g_n)_{n\in \mathbb{N}} \subseteq C([0,T];V)$ with $\|g_n \|_{\infty} \rightarrow 0$ that 
	\[ \int_0^T \langle Y(s),g_{n}(s)\rangle \ud s \rightarrow 0 \qquad \text{in probability as }n \to\infty;\]
	\item for every $v \in \mathcal{D}(A^*)$ and $t\in [0,T]$, $P$-almost surely, we have
	\begin{align}\label{eq.def-weak-eq}
	\langle Y(t), v\rangle = \langle Y_0,v\rangle +\int_0^t \langle Y(s),A^*v\rangle \ud s+L(t)(B^*v).
	\end{align}
\end{enumerate} 
Theorem 4.3 in \cite{kumar_riedle} shows that there exists a weak solution 
of the stochastic Cauchy problem \eqref{SCP1} on an interval $[0,T]$ if and only if
the map $s \to T(s)B$ is stochastically integrable with respect to $L$ on $[0,T]$, in which case the solution is unique. Together with Lemma~\ref{le.whole-line-solution} below it follows that in this case the solution exists on each interval $[0,S]$ and is given by
\begin{align}\label{OUP}
Y(t) = T(t)Y_0 + \int_0^tT(t-s)B \ud L(s) \qquad \text{ for all $t\ge 0$}.
\end{align}
It remains to establish the following: 
\begin{lemma} \label{le.whole-line-solution}
If there exists a weak solution for the stochastic Cauchy problem \eqref{SCP1} on $[0,T]$ for some $T>0$, then there exists  a weak solution  on $[0,S]$
for any $S>0$. 
\end{lemma}
\begin{proof} 
Choose $M \in \N$ such that $S/M \le T$ and define the cylindrical random variable
	\[Z \colon V \rightarrow L_P^0(\Omega; \mathbb{R}), \qquad Zv :=\int_0^SB^*T^*(s)v\ud L(s).\]
By \cite[Lemma 5.4]{OU} and the semigroup property,  we obtain for each $v\in V$ that
	\begin{align*}
	\phi_{Z}(v) &=  \exp \left(\int_0^S\Psi(B^*T^*(s)v)\ud s\right)\\
	&= \prod_{i=0}^{M-1}\exp \left(\int_{\frac{iS}{M}}^{\frac{(i+1)S}{M}}\Psi(B^*T^*(s)v)\ud s\right)\\
	&= \prod_{i=0}^{M-1}\exp \left(\int_{0}^{\frac{S}{M}}\Psi\left(B^*T^*\left(s+\frac{iS}{M}\right)v\right)\udd s\right)\\
	&= \prod_{i=0}^{M-1}\exp \left(\int_{0}^{\frac{S}{M}}\Psi\left(B^*T^*(s)T^*\left(\frac{iS}{M}\right)v\right)\udd s\right). \numberthis \label{ch_fn_ZA}
	\end{align*}
On the other hand, stochastic integrability of the map $s\mapsto T(s)B$ in $[0,T]$ implies that there exists a genuine probability distribution $\theta$
with characteristic function
	\begin{align} \label{ch_fn_nui}
	\phi_{\theta}(v)= \exp\left(\int_{0}^{\frac{S}{M}}\Psi\left(B^*T^*(s)v\right)\ud s\right).
	\end{align}	
If for each $i \in \{0, \ldots , M-1\}$, the image measure $\theta \circ T\left(\tfrac{iS}{M}\right)^{-1}$ is denoted by $\lambda_{i}$  and $\lambda:= \lambda_0 * \cdots * \lambda_{M-1}$, then it follows from \eqref{ch_fn_ZA} and \eqref{ch_fn_nui} that
	\[  \phi_{\lambda}(v) 
	=\prod_{i=0}^{M-1} \phi_{\theta}\left(T^\ast\left(\tfrac{iS}{M}\right)v\right)
	=
	\phi_{Z}(v) \qquad \text{for all  } v \in V.\]
Theorem IV.2.5 in \cite{vakhania} implies that $Z$ is induced by a genuine $V$-valued random variable.
	Hence $s \mapsto T(s)B$ is stochastically integrable in $[0,S]$ which completes the proof by Theorem 4.3 in \cite{kumar_riedle}.
\end{proof}

In the rest of this article we assume that the map  $s \mapsto T(s)B$ is stochastically integrable with respect to $L$  in $[0,T]$ for some (and hence each) $T >0$. In this case $\int_0^t T(s)B\, \ud L(s)$ is an infinitely divisible, $V$-valued 
random variable and we define  
\begin{align*}
\nu_t:=\mathscr{L}\left(\int_0^t T(s)B \ud L(s) \right)
\qquad\text{for all }t\ge 0.
\end{align*}
If $(a,Q,\mu)$ denotes the cylindrical characteristics of $L$, then the usual characteristics $(c_t,S_t,\xi_t)$ of $\nu_t$ is given by
\begin{align}
\label{ct_defn}
\langle c_t,v\rangle &=\int_0^t a(B^*T^*(s)v)\ud s +\int_V\langle h, v\rangle \big(\1_{B_V}(h)-\1_{B_{\mathbb{R}}}(\langle h, v\rangle)\big)\,\xi_t(\udd h),\\
\langle v, S_{t}v\rangle & = \int \limits_{0}^{t} \langle B^*T^*(s)v, QB^*T^*(s)v\rangle\ud s, \\
\xi_t & = ( \text{leb}\otimes\mu )\circ \chi_t^{-1} \quad \text{on } \mathcal{Z}(V), 
\label{xit_defn}
\end{align}
where
$\chi_t \colon [0,\infty)\times U \rightarrow V$ is defined by $\chi_t(s,u):=\1_{[0,t]}(s)T(s)Bu$.

A probability measure $\nu$ on $\Borel(V)$  is called a \emph{stationary measure}  for the process $(Y(t): t\geq 0)$ defined in \eqref{OUP} if it satisfies
	\begin{align}\label{de.stationary_measure}
	\nu = T_t\nu *\nu_t \qquad \text{for all  }t \ge 0,
	\end{align}
where $T_t\nu$ denotes the forward measure $\nu\circ (T(t))^{-1}$. 
Equivalently, a measure satisfying \eqref{de.stationary_measure} is also called an \emph{operator self-decomposable measure}. 

A stationary measure can also be defined as the invariant measure for the generalised Mehler semigroup  of the process $Y$. The concept of a generalised Mehler semigroup has been studied in detail in \cite{rockner_bogachev} for the Gaussian case and \cite{rockner_fuhrman} for the non-Gaussian case.                                                                                                                                                                                                                                                                                                                                                                                                                                                                                                                                                                                                                                                                                                                                                                                                                                                                                                                                                                                                                                                                                                                                                                                                                                                                                                                                                                                                                                                                                                                                                                                                                                                                                                                                                                                                                                                                                                                                                                                                                                                                                                                                                                                                                                                                                                                                                                                                                                                                                                                                                                                                                                                                                                                                                                                                                                                                                                                                                                                                                                                                                                                                                                                                                                                                                            
First, we need to know that the family $(\nu_t:t\ge 0)$ defines a skew-convolution semi-group:
\begin{lemma}\label{le.skew_conv}
The family  $(\nu_t:t\ge0)$ of probability measures on $\Borel(V)$ satisfies 
	\begin{align}\label{eq.skew_conv}
	\nu_{t+s} = T_t\nu_s *\nu_t \qquad \text{   for all  } s, t \ge 0.
	\end{align}
\end{lemma}
\begin{proof} Let $\phi_{T_t\nu_s*\nu_t}\colon V \to \C$ denotes the characteristic function of the probability measure $T_t\nu_s*\nu_t$. For each $v\in V$ and $s,t\ge 0$, we obtain, 
	\begin{align*}
	\phi_{T_t\nu_s*\nu_t}(v)
	&= \phi_{\nu_s}(T^*(t)v)\phi_{\nu_t}(v)\\
	& = \exp\left(\int_0^s\Psi\left(B^*T^*(r +t)v\right)\ud r +\int_0^t\Psi\left(B^*T^*(r)v\right)\ud r\right)\\
	& = \exp\left(\int_0^{t+s}\Psi\left(B^*T^*(r)v\right)\ud r \right)\\
	& = \phi_{\nu_{t+s}}(v),
	\end{align*}
	which establishes \eqref{eq.skew_conv}.
\end{proof}
The generalised Mehler semigroup $(P_t:t\ge 0)$ for the family $(\nu_t:t\ge 0)$ is defined by 
\[P_t\colon B_b(V) \to B_b(V), \qquad P_tf(v) = \int_Vf\big(T(t)v+h\big)\,\nu_t(\udd h),\]
where $B_b(V)$ denotes the space of all bounded and Borel measurable functions on $V$.
The generalised Mehler semigroup is a semigroup by \cite[Prop.\ 2.2]{rockner_bogachev}  because $(\nu_t:t\ge 0)$ is a skew-convolution semigroup by Lemma \ref{le.skew_conv}. A measure $\nu$ is called an {\em invariant measure for the transition semigroup $(P_t: t\ge 0)$} if for all $f \in B_b(V)$ and $t\ge 0$,
\begin{align}
\int_VP_tf(v)\,\nu(\udd v)=\int_Vf(v)\,\nu(\udd v).
\end{align}
The following equivalence result is from \cite[Theorem 2.1]{app_infinitesimal}, whose proof identically applies to the cylindrical case. 
\begin{thm} The following are equivalent for a measure $\nu$ on $\Borel(V)$:
\begin{enumerate}[{\rm(a)}]
		\item $\nu$ is a stationary measure for the process \eqref{OUP}, i.e.\ it satisfies \eqref{de.stationary_measure};
		\item $\nu$ is an invariant measure for the generalised Mehler semigroup $(P_t: t\ge 0)$. 
		\item if $Y_0$ has probability distribution $\nu$ then the  process $(Y(t):t\ge 0)$ defined in \eqref{OUP} is strictly stationary.
	\end{enumerate}
\end{thm}
\begin{proof}
See Theorem 2.1 in \cite{app_infinitesimal}. 
\end{proof}

A natural candidate for a stationary measure is the limit of $\nu_t$ in ${\mathcal M}(V)$ as $t\to\infty$. The following result relates the limit to the 
stochastic integral:
\begin{lemma}\label {eqv} The following conditions are equivalent:
	\begin{enumerate}[{\rm (a)}]
		\item $(\nu_t:\, t\ge 0)$  converges in ${\mathcal M}(V)$  as $t \to\infty$;
		\item $\left(\int_0^t T(s)B \ud L(s):\, t\ge 0\right)$ converges 
		in $L_P^0(\Omega;V)$ as $t \to\infty$.
	\end{enumerate}
In this case, the probability distribution of the limit in (b) coincides with the limit in (a). 
\end{lemma}
\begin{proof} 
Note that  the process $\big(\int_0^tT(s)B\ud L(s):\, t \geq 0\big)$ has independent increments which follows from the definition of the stochastic integral as a limit of stochastic integrals of simple integrands. Consequently, Lemma A.2.1 in   \cite{jurek_vervaat} guarantees that convergence in probability and weak convergence coincide.
\end{proof}

\begin{lemma} \label{stationarym} If $(\nu_t:\, t\ge 0)$  converges to $\nu$ in ${\mathcal M}(V)$ as $t \rightarrow \infty$, then it follows that:  
\begin{enumerate}[{\rm (a)}]
	\item the limit $\nu$ is a stationary measure for the process \eqref{OUP};
	\item any stationary measure $\lambda$ for \eqref{OUP} has the form $\lambda =\beta*\nu$, where $\beta$ is a probability measure satisfying  $\beta = T_t\beta$ for all $t\geq 0$. 
	\end{enumerate}
\end{lemma}                                                                                                                                    

\begin{proof}
 Lemma \ref{le.skew_conv} guarantees $\nu_{t+s} = T_t\nu_s *\nu_t$ for any $s,t \ge 0$. By taking limit as $s \to \infty$, we obtain
	\[\nu =T_t\nu*\nu_t \qquad \text{for all  }t \ge 0,\]
	which proves (a). To establish (b), we follow the arguments in \cite[Prop. 3.2]{michalik}. Let $\lambda$ be an invariant measure for \eqref{OUP} and  $(t_n)_{n\in\N}\subseteq \R^{+}$ a sequence converging to $\infty$. By the definition of the invariant measure,  we have
	\begin{align}\label{eq.lambda_is_invariant}
	\lambda = T_{t_n}\lambda*\nu_{t_n} \qquad \text{ for all  } n \in \N.
	\end{align} 
Since $(\nu_{t_n}:\, n\in\N)$ is relatively compact in $\mathcal{M}(V)$, and $\{\lambda\}$ is trivially relatively compact,  Theorem III.2.1 in \cite{partha} guarantees that the sequence $(T_{t_n}\lambda: n \in \N)$ is relatively compact in $\mathcal{M}(V)$. As a consequence of infinite divisibility of distributions $\nu$ and $\nu_t$, we obtain $\phi_{\nu}(v)\neq 0$ and $\phi_{\nu_t}(v) \neq 0$ for all $v\in V$. It follows by \eqref{eq.lambda_is_invariant} that, 
	\[ \phi_{T_{t_n}\lambda}(v) = \frac{\phi_{\lambda}(v)}{\phi_{{\nu}_{t_n}}(v)}\rightarrow \frac{\phi_{\lambda}(v)}{\phi_{\nu}(v)} \qquad \text{as  } n \to \infty.\]
Hence, since  $(t_n)_{n\in\N}$ is an arbitrary sequence,  Lemma VI.2.1 in  \cite{partha} implies that  $(T_t\lambda:\, t\ge 0)$ converges weakly to some probability measure $\beta$, and thus we obtain $\lambda = \beta *\nu$ by \eqref{eq.lambda_is_invariant}. Using that both $\lambda$ and $\nu$ are stationary measures for \eqref{OUP}, we conclude
	\[ \beta * \nu = \lambda = T_t\lambda*\nu_t = T_t(\beta*\nu)*\nu_t  = T_t\beta * (T_t\nu *\nu_t)=T_t\beta*\nu.\]
	Consequently, $\phi_{\beta}(v)\phi_{\nu}(v) = \phi_{T_t\beta}(v)\phi_{\nu}(v)$ for all $v\in V.$ Since $\phi_{\nu}(v) \neq 0$ for all $v\in V$, we derive $\phi_{\beta}(v)= \phi_{T_t\beta}(v)$ implying $\beta = T_t \beta$.
\end{proof}         
   
By Lemma \ref{stationarym}, if the  sequence $(\nu_t:\, t\ge 0)$ converges in ${\mathcal M}(V)$ then its limit is a stationary measure. Thus,  conditions for the  convergence of $(\nu_t:\, t\ge 0)$ provide conditions for the existence of a stationary measure.   Later in the case of stable semigroups, we will see that the converse implication is also true, i.e.\ the existence of a stationary measure  implies  convergence of $(\nu_t:\, t\ge 0)$. 

\begin{thm} \label{condiff}
The sequence $(\nu_t:\, t\ge 0)$ converges in ${\mathcal M}(V)$ as $t\to\infty$ 
if and only if the characteristics of $\nu_t$ defined in \eqref{ct_defn} - \eqref{xit_defn} 
satisfy the following conditions: 
	\begin{enumerate}[(a)]
\item[{\rm (a)}] \itemEq{\text{The limit of $c_t$ exists in $V$ as $t\to\infty$  where 
		$c_t$ is defined in \eqref{ct_defn}};
		\label{condition_cinfty_exist}}
\item[{\rm (b)}] \itemEq{\int_0^{\infty} \mathrm{tr} \big[T(s)BQB^*T^*(s)\big]\ud s <\infty; \label{finite_trace_condition}}
\item[{\rm (c)}]  \itemEq{ \label{condition2_levy_measure_inv_iff}	 \sup_{n\geq 1} \int_0^{\infty}\int_U\left( \sum_{k=1}^n \langle u, B^*T^*(s)h_k\rangle^2 \wedge 1 \right) \mu (\udd u)\ud s < \infty;}
\item[{\rm (d)}]  \itemEq{ \label{condition_levy_measure_inv_iff}	\limsup_{m\to\infty} \sup_{n\geq m} \int_0^{\infty}\int_U\left( \sum_{k=m}^n \langle u, B^*T^*(s)h_k\rangle^2 \wedge 1 \right) \mu (\udd u)\ud s=0.}
	\end{enumerate}
\end{thm}
For the proof of Theorem \ref{condiff} we need some results on the 
cylindrical measure $\xi_\infty$ defined by 
\begin{align}\label{eq.def-eta}
\xi_\infty:=( \text{leb}\otimes \mu ) \circ \chi_{\infty}^{-1}
\colon  \mathcal{Z}(V)\to [0,\infty], 
\end{align}
where $\chi_{\infty}\colon [0,\infty) \times U \to V$ is defined by $\chi_{\infty}(s,u):= T(s)Bu$. The canonical projection is denoted by $\pi_n$, i.e.\
\begin{align*}
\pi_n\colon V\to\ V, \qquad \pi_n(v)=\sum_{k=1}^n \scapro{v}{h_k}h_k, 
\end{align*}
where $(h_k)_{k\in\N}$ is an orthonormal basis of $V$. 

\begin{lemma}\label{eta_is_cyl_levy} If \eqref{condition2_levy_measure_inv_iff} holds, then it follows that
\begin{align*}
	\int_V\left(\scapro{h}{v}^2 \wedge 1\right) \,\xi_\infty(\udd h)<\infty
\qquad\text{for all }v\in V.
\end{align*}
\end{lemma}
\begin{proof}
For  any $v \in V$ and $n \in \N$ we obtain by the Cauchy-Schwarz inequality that
	\begin{align*}
& \int_0^{\infty}\int_U\left(  \scapro{T(s)Bu}{\pi_n(v)}^2 \wedge 1 \right)\, \mu (\udd u)\ud s\\
&\qquad= \int_0^{\infty}\int_U\left( \left( \sum_{k=1}^n \scapro{T(s)Bu}{h_k}\scapro{v}{h_k}\right)^2
\wedge 1\right)\,\mu (\udd u)\ud s\\
&\qquad  \le \max\{1,\|v\|^2\}\int_0^{\infty}\int_U\left( \sum_{k=1}^n \langle u, B^*T^*(s)h_k\rangle^2 \wedge 1 \right) \mu (\udd u)\ud s.
	\end{align*}
Assumption \eqref{condition2_levy_measure_inv_iff} implies  	\begin{align}\label{eq3.cyl_levy_pf}
	\sup_{n\ge 1} \int_0^{\infty}\int_U\left(  \scapro{T(s)Bu}{\pi_n(v)}^2 \wedge 1 \right)\, \mu (\udd u)\ud s< \infty. 
\end{align}
	Since for any sequence $(u_n)_{n\in\N}\subseteq U$ satisfying $u_n \to u$ in $U$, the finite measures $(|\beta|^2 \wedge 1)\,(\mu \circ \scapro{\cdot}{u_n}^{-1})$ converge weakly to $(|\beta|^2 \wedge 1)\, (\mu \circ \scapro{\cdot}{u}^{-1})$ according to \cite[Lemma 4.4]{Riedle_infinitely}, it follows that
	\begin{align*}
	&\lim_{n\to\infty}\int_U\left(  \scapro{T(s)Bu}{\pi_n(v)}^2 \wedge 1 \right)\,  \mu (\udd u)\\
	& \qquad =\lim_{n\to\infty} \int_{\R}\left(|\beta|^2 \wedge 1 \right)\, \left(\mu \circ \scapro{\cdot}{B^*T^*(s)\pi_n(v)}^{-1}\right)(\udd \beta)\\
	& \qquad = \int_{\R}\left(|\beta|^2 \wedge 1 \right)\, \left(\mu \circ \scapro{\cdot}{B^*T^*(s)v}^{-1}\right)(\udd \beta)\\
	& \qquad =\int_U\left(  \scapro{T(s)Bu}{v}^2 \wedge 1 \right)\, \mu (\udd u).	\end{align*} 
Consequently, Fatou's lemma guarantees for each $v \in V$ that
	\begin{align*}
	\int_V\left(\scapro{h}{v}^2 \wedge 1\right)\xi_\infty(\udd h) & = \int_0^{\infty}\int_U\left(  \scapro{T(s)Bu}{v}^2 \wedge 1 \right) \mu (\udd u)\ud s\\
	& \le \liminf_{n\to\infty} \int_0^{\infty}\int_U\left(  \scapro{T(s)Bu}{\pi_n(v)}^2 \wedge 1 \right) \mu (\udd u)\ud s. 
	\end{align*}
Applying \eqref{eq3.cyl_levy_pf} completes the proof.
\end{proof}

\begin{lemma}\label{le.eta_continuity_pi_n}
	If \eqref{condition2_levy_measure_inv_iff} holds, then the mapping 
	\begin{align*}
	f \colon V \to \C, \qquad f(v):=  \int_V \left( \cos(\langle h,v\rangle) -1\right)\,\xi_\infty(\udd h)
	\end{align*}
	satisfies $f(\pi_nv) \to f(v)$ as $n \to \infty$ for each $v\in V$. 
\end{lemma}
\begin{proof} We first note that by monotone convergence theorem and \eqref{condition2_levy_measure_inv_iff}, it follows  that
	\begin{align*}
	&\int_0^{\infty}\sup_{n\geq 1} \int_U\left( \sum_{k=1}^n \langle u, B^*T^*(s)h_k\rangle^2 \wedge 1 \right)\, \mu (\udd u)\ud s\\ 
	& \qquad  =\sup_{n\geq 1} \int_0^{\infty}\int_U\left( \sum_{k=1}^n \langle u, B^*T^*(s)h_k\rangle^2 \wedge 1 \right)\, \mu (\udd u)\ud s < \infty. \numberthis \label{condition2_levy_mct}
	\end{align*}
Let $v\in V$ be fixed and define for each $n\in \N$ the function
	\begin{align*}
	r_n\colon [0,\infty) \to \C, \qquad r_n(s):=\int_U\left(\cos\left(\scapro{T(s)Bu}{\pi_nv}\right)-1\right)\, \mu(\udd u).
	\end{align*}
It follows that 
	\begin{align}\label{eq.apply-Lebesgue-here}
	f(\pi_nv) & = \int_0^{\infty}\int_U\left(\cos (\scapro{T(s)Bu}{\pi_nv})-1\right)\,\mu(\udd u)\ud s = \int_0^{\infty} r_n(s)\ud s.
	\end{align}
Define the bounded and continuous function 
\begin{align*}
g\colon \R \to \C, \qquad g(\beta) = \left\{ \begin{array}{ll}
\frac{\cos(\beta)-1}{\beta^2 \wedge 1}, & \text{if  } \beta\ne 0,\\
 -\frac{1}{2}, & \text{if  }\beta =0. \end{array}\right.
\end{align*}
Lemma 4.4 in \cite{Riedle_infinitely} implies that 
	\begin{align*}
	\lim_{n\to\infty}r_{n}(s) 
	& =\lim_{n\to\infty}\int_Ug\left( \scapro{u}{B^*T^*(s)\pi_nv}\right)\left(  \scapro{u}{B^*T^*(s)\pi_nv}^2 \wedge 1 \right) \,\mu (\udd u)\\
	&  =\lim_{n\to\infty} \int_{\R}g(\beta)\left(|\beta|^2 \wedge 1\right) \, \left(\mu \circ \scapro{\cdot}{B^*T^*(s)\pi_nv}^{-1}\right)(\udd \beta)\\
	&  = \int_{\R}g(\beta)\left(|\beta|^2 \wedge 1\right) \, \left(\mu \circ \scapro{\cdot}{B^*T^*(s)v}^{-1}\right)(\udd \beta)\\
	&  =\int_U \left(\cos (\scapro{T(s)Bu}{v})-1\right) \,\mu(\udd u).
	\numberthis \label{eq43.cyl_levy_pf}
	\end{align*}
For each $n \in \N$ and $s\ge 0$, we obtain by the Cauchy-Schwarz inequality that 
	\begin{align*}
	|r_{n}(s)| &= \left|\int_Ug\big(\scapro{u}{B^*T^*(s)\pi_nv})\big)\left(  \scapro{u}{B^*T^*(s)\pi_nv}^2 \wedge 1 \right)\, \mu(\udd u)\right|\\
	& \le \|g\|_{\infty}\int_U\left(\langle \pi_nT(s)Bu,v\rangle^2 \wedge 1\right)\, \mu(\udd u)\\
	& \le \|g\|_{\infty}\max\{1,\|v\|^2\}\int_U\left(\sum_{k=1}^n\scapro{T(s)Bu}{h_k}^2 \wedge 1\right)\,\mu(\udd u). \label{eq51.alpha_bdd}\numberthis
	\end{align*}
	In view of \eqref{condition2_levy_mct}, \eqref{eq43.cyl_levy_pf} and \eqref{eq51.alpha_bdd}, applying Lebesgue's dominated convergence theorem to
 \eqref{eq.apply-Lebesgue-here} completes the proof.
\end{proof}
\begin{lemma}\label{le.levy_ext_eqv} The following conditions are equivalent:
	\begin{enumerate}[(a)]
\item the cylindrical measure $\xi_\infty$ defined 
in \eqref{eq.def-eta} extends to a L\'evy measure on $\Borel(V)$.
		\item Conditions \eqref{condition2_levy_measure_inv_iff} and \eqref{condition_levy_measure_inv_iff}  are satisfied.
	\end{enumerate}
\end{lemma}
\begin{proof}(a) $\Rightarrow$ (b). The result follows by making use of the definition of a L\'evy measure, monotone convergence theorem and Lebesgue's theorem.\\
	(b) $\Rightarrow$ (a).
	For any $N\in \N$ let $\rho_N : = (\xi_\infty +\xi_\infty^{-})\circ\pi_N^{-1}$, where $\xi_\infty^{-}(C):=\xi_\infty(-C)$ for all $C\in \mathcal{Z}(V)$. Then $\rho_N$ extends to a measure as $\pi_N$ is Hilbert-Schmidt, and satisfies 
	\begin{align*}
\int_V\left(\|v\|^2 \wedge 1\right) \,\rho_N(\udd v)
=   2 \int_V\left(\sum_{k=1}^N\scapro{v}{h_k}^2 \wedge 1\right)\, \xi_\infty(\udd v)<\infty.
	\end{align*}
Consequently, $\rho_N$ is a genuine L\'evy measure on $\Borel (V)$. The L\'evy-Khinchine Theorem implies that there exists an infinitely divisible probability measure $\theta_N$ on $\Borel(V)$ with characteristic function
	\begin{align*}
	\phi_{\theta_N}\colon V \to \C, \qquad \phi_{\theta_N}(v):= \exp \left(\int_V \left(\cos \left(\scapro{v}{h}\right)-1\right)\, \rho_N(\udd h) \right).
	\end{align*}
By an application of  the inequality $1-\cos \beta \le 2(\beta^2 \wedge 1)$ for all $\beta \in \R$, it follows that for every $v\in V$ we have
	\begin{align*}
	1-\phi_{\theta_N}(v) &= 1-\exp \left(\int_V\cos (\scapro{v}{h}-1)\,\rho_N(\udd h)\right)\\
	& \le \int_V(1-\cos (\scapro{v}{h})\,\rho_N(\udd h)
 \le 2\int_V (\scapro{h}{v}^2 \wedge 1)\,\rho_N(\udd h).
	\end{align*}
By denoting the density of the standard normal distribution on $\Borel(\R^m)$ by $g_m$, 
we obtain for every $m,n \in \N$ with $m\le n$ and $N \in \N$  that
	\begin{align*}
	&\int_{\R^{n-m+1}} \left(1-\text{Re} \phi_{\theta_N}(\beta_mh_m+\cdots+\beta_nh_n)\right)g_{n-m+1}(\beta_m,\ldots,\beta_n)\ud \beta_m\cdots\ud \beta_n\\
	&  \leq 2\int_{\R^{n-m+1}}\int_V \left(\left|\sum_{k=m}^n\beta_k\scapro{h}{h_k}\right|^2\wedge 1\right)\rho_N(\udd h)g_{n-m+1}(\beta_m,\ldots,\beta_n)\ud \beta_m\cdots\ud \beta_n\\
	&  \leq 2\int_V \left(\left(\int_{\R^{n-m+1}}\left|\sum_{k=m}^n\beta_k\scapro{h}{h_k}\right|^2g_{n-m+1}(\beta_m,\ldots,\beta_n)\ud \beta_m\cdots\ud \beta_n \right)\wedge 1\right)\rho_N(\udd h)\\
	&  = 2\int_V\left(\sum_{k=m}^n\scapro{h}{h_k}^2\wedge 1\right)\rho_N(\udd h)\\
	&  = 2\int_V\left(\sum_{k=m}^n\scapro{\pi_Nh}{h_k}^2\wedge 1\right)\,(\xi_\infty+\xi_\infty^{-})(\udd h)\\
	&  \leq 4\int_V\left(\sum_{k=m}^n\scapro{h}{h_k}^2\wedge 1\right)\,\xi_\infty(\udd h)\\
	&  = 4\int_0^{\infty}\int_U\left( \sum_{k=m}^n \langle u, B^*T^*(s)h_k\rangle^2 \wedge 1 \right) \mu (\udd u)\ud s. \label{eq.condition-1-Para}\numberthis
	\end{align*}
Furthermore, for fixed $n \in \N$ and for each $N \in \N$, define the function
	\[\psi_N\colon \R^n \to \C, \qquad \psi_N(\beta_1, \ldots, \beta_n)=  \phi_{\theta_N}\left(\beta_1h_1+\cdots +\beta_nh_n\right).\]
For each $\beta=(\beta_1, \ldots, \beta_n)\in\R^n$ we have
	\begin{align}
	\psi_N(\beta) = \exp \left(\int_V\left(\cos \left(\scapro{\beta_1h_1+\cdots +\beta_nh_n}{h}\right)-1\right)\rho_N(\udd h)\right) .\label{eq.psi_equi_aux1}
	\end{align}
For fixed $n\in\N$, Lebesgue's theorem on dominated convergence implies
	\begin{align*}
	&\left|\int_V\left(\cos \left(\scapro{\beta_1h_1+\cdots +\beta_nh_n}{h}\right)-1\right)\rho_N(\udd h)\right|\\
	& \qquad \leq \int_V\left|\cos \left(\scapro{\beta_1h_1+\cdots +\beta_nh_n}{h}\right)-1\right|\rho_N(\udd h)\\
	& \qquad  \leq 2 \int_V\left(\scapro{\beta_1h_1+\cdots +\beta_nh_n}{h}^2 \wedge 1\right)\rho_N(\udd h)\\
	& \qquad =  4 \int_V\left(\scapro{\beta_1h_1+\cdots +\beta_nh_n}{\pi_Nh}^2 \wedge 1\right)\xi_\infty(\udd h)\\
	& \qquad \leq  4 \int_V\left(|\beta|^2\sum_{k=1}^n\scapro{h_k}{h}^2 \wedge 1\right)\xi_\infty(\udd h)\\
	& \qquad=4 \int_{\R^n}(|\beta|^2|\alpha|^2 \wedge 1) \left(\xi_\infty \circ \pi_{h_1,\ldots,h_n}^{-1}\right)(\udd \alpha)\\
	& \qquad \to 0 \quad \text{ as $|\beta| \to 0$ and uniformly in $N\in\N$.}
	\end{align*}
It follows from \eqref{eq.psi_equi_aux1} that the  family $(\psi_N:N\in\N)$ is equicontinuous at the origin. Applying Condition \eqref{condition_levy_measure_inv_iff} to \eqref{eq.condition-1-Para} and the derived equicontinuity enable us to 
deduce from Lemma VI.2.3 in \cite{partha} that the family $\{\theta_N: N\in\N\}$ is relatively compact in $\mathcal{M}(V)$. Furthermore, Lemma   \ref{le.eta_continuity_pi_n} implies for each $v \in V$ that
	\begin{align*}
	\lim_{N\to\infty}\phi_{\theta_N}(v) &= \lim_{N\to\infty}\exp \left(\int_V \left(\cos \left(\scapro{v}{h}\right)-1\right)\rho_N(\udd h) \right)\\
	&=\lim_{N\to\infty}\exp \left(\int_V \left(\cos \left(\scapro{\pi_Nv}{h}\right)-1\right)(\xi_\infty+\xi_\infty^{-})(\udd h) \right)\\
	&=\exp \left(\int_V \left(\cos \left(\scapro{v}{h}\right)-1\right)(\xi_\infty+\xi_\infty^{-})(\udd h) \right).
	\end{align*}
	It follows by \cite[Lemma VI.2.1]{partha} that $(\theta_N)_{N\in\N}$ converges weakly to an infinitely divisible probability measure $\theta$ and the characteristic function $\phi_\theta$ of $\theta$ is given by 
	\[\phi_{\theta}(v) = \exp \left(\int_V \left(\cos \left(\scapro{v}{h}\right)-1\right)(\xi_\infty+\xi_\infty^{-})(\udd h) \right)\qquad \text{for all }v\in V.\]
	Consequently, $\xi_\infty+\xi_\infty^{-}$ extends to the L\'evy measure of $\theta$. 
	Since
	\[\xi_\infty(C) \le \xi_\infty(C)+\xi_\infty^{-}(C) \qquad \text{ for all  } C\in \mathcal{Z}(V), \] 
	Theorem 3.4 in \cite{OU} implies that $\xi_\infty$ extends to a L\'evy measure on $\Borel(V)$, which completes the proof. 
%
\end{proof}

\begin{proof}[Proof of Theorem \ref{condiff}]
	\emph{Sufficiency}:  suppose that \eqref{condition_cinfty_exist}--\eqref{condition_levy_measure_inv_iff} hold. We first show that the family $(\nu_t:t\ge 0)$ of infinitely divisible probability measures with characteristics $(c_t,S_t,\xi_t)$ is relatively compact in $\mathcal{M}(V)$, for which we use 
the compactness criterion for infinitely divisible probability measures as given  in \cite[Th. VI.5.3]{partha}. We only need to show that
	the set $(\xi_t:t\geq 0)$ restricted to the complement of any neighbourhood of the origin is relatively compact and the operators $R_t\colon V \rightarrow V$ defined by
	\beq\langle R_tv,v\rangle :=\langle S_tv,v\rangle+\int_{\|h\| \leq 1}\langle v,h\rangle^2 \, \xi_t(\udd h) \label{opTt}\eeq
	satisfy 
\begin{align}
 \sup \limits_{t\geq 0} \sum \limits_{k=1}^{\infty}\langle R_th_k,h_k\rangle &<\infty \label{eq.condition-para-finite},\\
 \lim_{m \rightarrow \infty}\sup_{t\geq 0} \sum \limits_{k=m}^{\infty}\langle R_th_k,h_k\rangle &=0. \label{eq.condition-para-0}
\end{align}
For any cylindrical set $C \in  \mathcal{Z}(V)$ and $t\ge 0$, we have
	\begin{align*}
	\xi_t(C) = \int_{0}^t\int_U\1_{C}(T(s)Bu)\,\mu(\udd u)\ud s
	& \leq \int_{0}^\infty\int_U\1_{C}(T(s)Bu)\, \mu(\udd u)\ud s = \xi_\infty(C). 
	\end{align*}
Since  $\Borel(V)$ is the $\sigma$-algebra generated by $\mathcal{Z}(V)$ and $\mathcal{Z}(V)$ is a $\pi$-system, we obtain $\xi_t \leq \xi_{\infty}$ on $\Borel(V)$ for all $t \geq 0$. Let $\xi_t^{c}$ and $\xi^{c}_{\infty}$ denote the restrictions of the measures $\xi_t$ and $\xi_{\infty}$ to the complement of a neighbourhood $V_1 \subseteq V$ of origin. By Lemma~\ref{le.levy_ext_eqv} and \cite[Prop 1.1.3]{linde}, the finite measure $\xi^{c}_{\infty}$ is a Radon measure and therefore, for each $\epsilon >0$  there exists a compact set $K \subseteq V_1^c$ such that $\xi^c_{\infty}(K^c) \le \epsilon$. As a consequence,
	\beq \xi_t^{c}(K^c) \leq \xi_{\infty}^{c}(K^c) \leq \epsilon \quad \text{for all  } t \ge 0, \eeq
	which implies that $(\xi_t: t \geq 0)$ restricted to the complement of any neighbourhood of the origin is relatively compact. Furthermore,  Lebesgue's theorem on dominated convergence and \eqref{finite_trace_condition} imply
	\begin{align*}
	&\lim_{m \rightarrow \infty}\sup_{t\geq 0} \int_0^t \sum_{k=m}^{\infty}\langle T(s)BQB^*T^*(s)h_k,h_k\rangle\ud s\\
	&\qquad = \lim_{m \rightarrow \infty}\int_0^{\infty} \sum_{k=m}^{\infty}\langle T(s)BQB^*T^*(s)h_k,h_k\rangle\ud s =0.\label{Ttcondb11} \numberthis
	\end{align*}
From Condition \eqref{condition_levy_measure_inv_iff} we deduce that 
	\begin{align*}
	\sup_{t\geq 0} \sum_{k=m}^{\infty}\int_{\|h\| \leq 1}\langle h,h_k\rangle^2 \,\xi_t(\udd h)& \le  \sup_{t\geq 0}\, \sup_{n \geq m}\int_{V}\left(\sum_{k=m}^{n}\langle h,h_k\rangle^2 \wedge 1\right)\,\xi_t(\udd h)\\
	&=  \sup_{t\geq 0}\, \sup_{n \geq m}\int_0^t\int_{U}\left(\sum_{k=m}^{n}\langle T(s)Bu, h_k\rangle^2 \wedge 1\right)\,\mu(\udd u)\ud s\\
	& =\sup_{n \geq m}\int_0^{\infty}\int_{U}\left(\sum_{k=m}^{n}\langle T(s)Bu,h_k\rangle^2 \wedge 1\right)\, \mu(\udd u)\ud s\\
	& \to 0 \qquad \text{as  } m \to \infty. \numberthis \label{Ttcondb21}
	\end{align*}
The limits in \eqref{Ttcondb11} and \eqref{Ttcondb21} show that Condition \eqref{eq.condition-para-0} is satisfied.
 Condition \eqref{eq.condition-para-finite} can be proved analogously using \eqref{condition2_levy_measure_inv_iff}, and thus Theorem VI.5.3 in \cite{partha}
imply that $(\nu_t:t\ge 0)$ is relatively compact. 

Since $t\mapsto$tr$(S_t)$ is increasing, Condition \eqref{finite_trace_condition} implies that  the operator 
	\[S_{\infty} := \int_0^{\infty}T(s)BQB^*T(s) \ud s,\]
	is well-defined and  
	\begin{align}\scapro{S_tv}{v} \to \scapro{S_{\infty}v}{v} \qquad  \text{ for all } v\in V. \label{St_to_sinfty}
	\end{align} 
Since $(\xi_t)_{t\ge 0}$ is an increasing family of L\'evy measures and  $\xi_t(A)$  increases to the L\'evy measure $\xi_{\infty}(A)$ for each $A \in \Borel(V)$, we obtain by \cite[Le. 3.3] {rockner_fuhrman} for each $v\in V$ that
	\begin{align}
\int_V
&	\left(e^{i\scapro{h}{v}}-1-i\scapro{h}{v}\1_{B_V}(v)\right) \xi_t(\udd h) \notag \\
&\qquad\qquad\qquad 	\to \int_V
	\left( e^{i\scapro{h}{v}}-1-i\scapro{h}{v}\1_{B_V}(v)\right) \xi_{\infty}(\udd h),  \label{weak_conv_K}
	\end{align}
as $t\to\infty$. 
It follows from \eqref{condition_cinfty_exist}, \eqref{St_to_sinfty} and \eqref{weak_conv_K} that the characteristic function $\phi_{\nu_t}$ of $\nu_t$ converges to the characteristic function $\phi_\nu$ of an infinitely divisible measure $\nu$ with characteristics $(c_\infty, S_{\infty}, \xi_{\infty})$.
Together with   relative compactness of $(\nu_t:\,t\ge 0)$,  Lemma VI.2.1 in \cite{partha} guarantees that   $(\nu_t:t\ge 0)$ converges  in $\mathcal{M}(V)$. 

\emph{Necessity}:  if $(\nu_t:\, t\ge 0)$ converges weakly as $t \to \infty$ then \eqref{condition_cinfty_exist}-\eqref{condition_levy_measure_inv_iff} follow by the compactness criterion of infinitely divisible probability measures in Hilbert spaces as applied before. 
\end{proof}

\begin{example}\label{ex.canonical_alpha_invariant} 
In this example, we assume that $U=V$ and $B={\rm Id} $ in equation \eqref{SCP1}. 
Let $L$ be the canonical $\alpha$-stable cylindrical L\'evy process for $\alpha \in (0,2)$, which is defined in \cite{Riedle_alpha_stable} by requiring that its characteristic function is of the form
\begin{align*}
\phi_{L(t)}(u)=\exp\left(-t\|u\|^\alpha\right)
\qquad\text{for all }u\in U, \, t\ge 0. 
\end{align*}
Assume that there exists an orthonormal basis $(e_k)_{k\in\N}$ of $U$ and an increasing sequence $(\lambda_k)_{k\in\N}\subseteq [0,\infty)$ with $T^\ast(t)e_k=e^{-\lambda_k t}e_k$ for all $t\ge 0$ and $k\in\N$. 
According to Theorem 4.1 in \cite{Riedle_alpha_stable}, the semigroup $(T(t))_{t\ge 0}$ is stochastically integrable on $[0,T]$ with respect to $L$ if and only if 
	\begin{align*}\
	\int_0^{T}\|T(s)\|^{\alpha}_{\text{HS}}\ud s< \infty,
	\end{align*}
where $\|\cdot \|_{\text{HS}}$ denotes the Hilbert-Schmidt norm. Theorem \ref{condiff} guarantees that there exists a stationary solution if and only if 
	\begin{align}\label{eq.int-stable-infty}
	\int_0^{\infty}\|T(s)\|^{\alpha}_{\text{HS}}\ud s< \infty.
	\end{align}
This can be seen by verifying Conditions \eqref{condition2_levy_measure_inv_iff} and \eqref{condition_levy_measure_inv_iff}	by similar arguments as exploited in the proof of Theorem 4.1 in \cite{Riedle_alpha_stable}. 

For example, a sufficient assumption for the validity of 
\eqref{eq.int-stable-infty} is
\begin{align*}
\sum_{k=1}^\infty \frac{1}{\lambda_k}<\infty. 
\end{align*}
This follows since Cauchy-Schwartz inequality implies
\begin{align*}
	\int_0^\infty \|T(s)\|^{\alpha}_{\text{HS}}\ud s & = 
\int_0^\infty \left(\sum_{k=1}^\infty e^{-2\lambda_ks}\right)^{\alpha/2}\ud s\\
	&\leq \int_0^\infty e^{-\frac{\alpha}{2}\lambda_1s}\left(\sum_{k=1}^\infty e^{-\lambda_ks}\right)^{\alpha/2}\ud s\\
	&\leq \left(\int_0^\infty e^{-\frac{\alpha}{2-\alpha}\lambda_1s}\ud s\right)^{\frac{2-\alpha}{2}}\left(\int_0^\infty \sum_{k=1}^\infty e^{-\lambda_ks}\ud s\right)^{\alpha/2}\\
	&=\left(\frac{2-\alpha}{\alpha\lambda_1}\right)^{\frac{2-\alpha}{2}} \left(\sum_{k=1}^{\infty}\frac{1}{\lambda_k}\right)^{\alpha/2}.
	\end{align*}
\end{example}

\begin{example}
More specifically, we consider the heat equation on a bounded domain ${\mathcal O}$ in $\R^d$ with smooth boundary for some $d\in\N$ in the setting of the previous Example \ref{ex.canonical_alpha_invariant}. In this case, the generator $A$ is given by the Laplace operator $\Delta$
on $U=L^2({\mathcal O})$ and $L$ is the canonical $\alpha$-stable cylindrical L\'evy process for $\alpha\in (0,2)$. Weyl's law guarantees that the eigenvalues of $A$ satisfy $\lambda_k = c_k k^{2/d}$ for all $k\in\N$, where $c_k \in [a,b]$ for some $a,b>0$. Example 4.2 in \cite{Riedle_alpha_stable} shows that there exists a solution if and only if $\alpha d<4$. We claim that the same condition is sufficient and necessary for the existence of a stationary solution. 

First, suppose $\alpha d <4$. By the integral test for convergence of series we obtain for each $s>0$ that
\begin{align*}
\|T(s)\|_{\text{HS}}^2 = \sum_{k=1}^{\infty}e^{-2c_ksk^{2/d}}
&\leq e^{-as}\sum_{k=1}^{\infty}e^{-ask^{2/d}}\\
& \leq e^{-as} \left(\int_0^{\infty}e^{-asx^{2/d}}\ud x\right)
=e^{-as}\frac{d \,\Gamma (\frac{d}{2})}{2a^{d/2}s^{d/2}}.
\end{align*}
Consequently, Condition \eqref{eq.int-stable-infty} is satisfied since 
\begin{align*}
\int_0^{\infty}	\|T(s)\|_{\text{HS}}^{\alpha} \ud s
& \le c_a \left(\int_0^{1}\frac{e^{-\frac{a\alpha s}{2}}}{s^{\frac{\alpha d}{4}}}\ud s+\int_1^\infty \frac{e^{-\frac{a\alpha s}{2}}}{s^{\frac{\alpha d}{4}}}\ud s\right)\\
&\le c_a \left(\int_0^{1}\frac{1}{s^{\frac{\alpha d}{4}}}\ud s+\int_1^\infty e^{-\frac{a\alpha s}{2}}\ud s\right)
 < \infty,
\end{align*}
where $c_a := \left(\frac{d \,\Gamma (\frac{d}{2})}{2a^{d/2}} \right)^{\alpha /2}$.
On the other hand, the integral test for series implies
\begin{align*}
\|T(s)\|_{\text{HS}}^2 =\sum_{k=1}^{\infty}e^{-2c_ksk^{d/2}}
& \geq -1+  \int_0^{\infty}e^{-2bsx^{d/2}}\ud x
=-1+\frac{d \,\Gamma (\frac{d}{2})}{2(2b)^{d/2}s^{d/2}},
\end{align*}
which results in  $\int_0^{1}	\|T(s)\|_{\text{HS}}^{\alpha} \ud s=\infty$ for $\alpha d\ge 4$.
\end{example}

\begin{remark}
 If $L$ is the genuine L\'evy process with (classical) characteristics $(b, Q, \mu)$, then the cylindrical characteristics of $L$ are given by $(a,Q,\mu)$ where
\begin{align}\label{eq.adef} a(u^*)= \langle b,u^*\rangle + \int_U  \langle u, u^* \rangle\left(\1_{B_{\R}}(\langle u, u^* \rangle) - \1_{B_U}(u) \right) \mu(\udd u).
\end{align}
Then for every $v\in V$, we have by \eqref{ct_defn} and \eqref{eq.adef}, 
\begin{align*}
\langle c_t,v\rangle 
&= \int_0^t a(B^*T^*(s)v)\ud s\\
& \qquad  +\int_0^t\int_U\langle u, B^*T^*(s)v\rangle \big(\1_{B_V}(T(s)Bu)-\1_{B_{\mathbb{R}}}(\langle u, B^*T^*(s)v\rangle)\big)\mu(\udd u)\ud s\\
&  =\int_0^t \langle T(s)Bb,v\rangle \ud s+ \int_0^t\int_U  \langle T(s)Bu, v \rangle\big(\1_{B_{V}}( T(s)Bu) - \1_{B_U}(u) \big) \mu(\udd u)\ud s.
\end{align*}
As a consequence, we observe that in this case, Theorem \ref{condiff} is equivalent to the well-known result from \cite{michalik}: 
the sequence $(\nu_t:\, t\ge 0)$ converges weakly if and only if the following conditions are satisfied:
\begin{enumerate}[(i)]
	\item  There exists
	\begin{align*}
	\label{levy_drift_cond}\lim_{t\rightarrow \infty}\left(\int_0^tT(s)Bb\ud s+ \int_0^t\int_U T(s)Bu  \big(\1_{B_V}(T(s)Bu)-1_{B_U}(u)\big)\,\mu(\udd u)\ud s\right) ;
	\end{align*} 
	\item \itemEq{\int_0^{\infty} \mathrm{tr}\left(T(s)BQB^*T^*(s)\right)\ud s <\infty;}
	\item 
	\itemEq{\label{levy_meas_cond_levy} \int_0^{\infty}\int_U\left( \|T(s)Bu\|^2 \wedge 1 \right) \mu(\udd u)\ud s < \infty.}
\end{enumerate}
The equivalence of \eqref{levy_meas_cond_levy} and the Conditions \eqref{condition2_levy_measure_inv_iff} and \eqref{condition_levy_measure_inv_iff} can be obtained by noting that in this case $\mu$ is a genuine L\'evy measure  and consequently $\xi_\infty = ( \text{leb}\otimes \mu ) \circ \chi_{[0,\infty)}^{-1}$ is also a genuine measure.  By Lemma \ref{le.levy_ext_eqv}, the Conditions \eqref{condition2_levy_measure_inv_iff} and \eqref{condition_levy_measure_inv_iff} are equivalent to the Condition that $\xi_\infty$ is a L\'evy measure which is equivalent to \eqref{levy_meas_cond_levy}.
\end{remark}

It is well known that in general an invariant measure is not necessarily unique. As in the case of genuine L\'evy processes, we obtain uniqueness if the semigroup 
 $(T(t))_{t\geq 0}$ on $V$ is stable, that is $T(t)v \to 0$ as $t\to \infty$  for each $v\in V$. 
\begin{thm}\label{th.stable_case} If the semigroup $(T(t))_{t\geq 0}$ is stable, then there exists a stationary measure $\nu$ for the process \eqref{OUP} if and only if  $(\nu_t)_{t\ge 0}$ converges weakly; in this case the limit of $(\nu_t)_{t\ge 0}$ equals the stationary measure. 
\end{thm}
\begin{proof}
Our claim in the cylindrical setting can be proved as in the classical situation;
see  \cite[Prop.\ 6.1]{michalik}.
\end{proof}

Combining Theorem \ref{th.stable_case} with Theorem \ref{condiff}, we obtain 
that, if the semigroup is stable, then  Conditions \eqref{condition_cinfty_exist}-\eqref{condition_levy_measure_inv_iff} of Theorem \ref{condiff} are necessary and sufficient for the existence of a stationary measure for the process \eqref{OUP}, 
which in this case is unique.

\section{The case of exponentially stable semigroups}\label{se.stable_case}
In general the conditions of Theorem \ref{condiff} may be difficult to verify in practice, in particular Condition \eqref{condition_cinfty_exist} for the drift component $c_t$.
If the semigroup is exponentially stable, i.e.\ there exists $C >1$ and $\lambda >0$ such that $\|T(t)\| \leq C e^{-\lambda t}$ for all $t \geq 0$, and $L$ is a genuine L\'evy process, then a sufficient condition  for the existence of stationary measure is that the L\'evy measure $\mu$ of $L$ satisfies the following simple condition
\begin{align}\label{eq.classical-log-condition}
\int_U \log^+ \|u\| \,\mu(\udd u) < \infty,
\end{align} 
where $\log^+ x := \log x$ if $x \geq 1$ and 0 otherwise; see \cite[Th.\ 6.7]{michalik}.
This condition is also necessary if $V$ is finite dimensional (see  \cite[Th.\ 4.3.17]{app_levy} and references therein) or if the semigroup $(T(t))_{t\in\R}$ is a group (see \cite[Prop.\ 6.8]{michalik}) but in general is not necessary (see  \cite[Ex.\ 3.15]{michalik_stationarity}). In the case of a semigroup $(T(t))_{t\geq 0}$ with spectral decomposition $T(t)e_k = e^{-\lambda_k} e_k$ (e.g.\ the heat semigroup) where the eigenvalues $(\lambda_k)$ satisfy some mild conditions (see \eqref{exp.condition}),  the following  weaker condition
\[\int_U \sup_{ n\in\N} \left(\frac{\log^+|\langle u,e_n\rangle|}{\lambda_n} \right) \,\mu(\udd u) < \infty,\] 
is shown in \cite{michalik_heat} to be both necessary and sufficient for the existence of a stationary measure when $L$ is a genuine L\'evy process. In the  main result of this section, we generalise this condition for the case of cylindrical L\'evy processes and give some examples.

Without loss of generality, we assume  $U = V$  and $B=\text{Id}$ in the rest of this section. We assume that $A$ is a self-adjoint strictly negative operator with compact resolvent. Consequently, $A$ has a purely point spectrum $(-\lambda_k)_{k\in\N}$, where
\begin{align}\label{eq.conditon-eigenvalues} 0< \lambda_1 \leq \lambda_2 \leq \cdots \quad\text{ and }\quad
  \lim_{k\rightarrow \infty} \lambda_k = \infty,
 \end{align}
and there is an orthonormal basis $(e_k)_{k\in\N}$ in $V$ consisting of  eigenvectors $e_k$ of $A$ corresponding to the eigenvalues $-\lambda_k$. Then $A$ is a generator of the $C_0$-semigroup $(T(t))_{t \geq 0}$ of  bounded linear operators on $V$, given by the formula:
\beq \label{heat_semigroup}
T(t)v = \sum_{k=1}^{\infty}e^{-\lambda_kt} \langle v, e_k\rangle e_k \quad \text{    for   } v \in V.
\eeq
Clearly, the semigroup $(T(t))_{t \geq 0}$ satisfies
\begin{align} \label{eq.heat_semigp_exp_stable}
\|T(t)v\| \leq e^{-\lambda_1 t} \|v\| \quad \text{ for all $v \in V$ and $t \ge 0$},
\end{align}
and is therefore exponentially stable.

\begin{thm}\label{th.heat_semigrp_iff}
Assume that the semigroup $(T(t))_{t \geq 0}$ is given by \eqref{heat_semigroup} and satisfies
	\begin{align} \label{exp.condition}
	\sum_{k=1}^{\infty}\frac{1}{\lambda_k} < \infty.
	\end{align}
	 Then the following are equivalent:
	\begin{enumerate}[(a)]
\item[{\rm (a)}] There exists a stationary measure for the process \eqref{OUP};
\item[{\rm (b)}] \begin{enumerate}[(i)]
			\item[{\rm (i)}] \itemEq{\label{logcond_sup_compact_semigroup} \sup_{n\geq 1} \int_U \max_{1\leq k\leq n} \left(\frac{\log^+|\langle u,e_k\rangle|}{\lambda_k} \right) \,\mu(\udd u) < \infty; }
\item[{\rm (ii)}]  \itemEq{\label{logcond_compact_semigroup}\limsup_{m\rightarrow \infty} \sup_{n\geq m} \int_U \max_{m\leq k\leq n} \left(\frac{\log^+|\langle u,e_k\rangle|}{\lambda_k} \right)\, \mu(\udd u) = 0. }
		\end{enumerate}
	\end{enumerate}
\end{thm}
\begin{proof} (b) $\Rightarrow$ (a).
	We show that the conditions in Theorem \ref{condiff} are satisfied.
The spectral representation of the semigroup \eqref{heat_semigroup}	 implies
	\begin{align*}
	\int_{0}^{\infty}\text{tr}(T(s)QT^*(s))\ud s 
  = \sum_{k=1}^{\infty}\scapro{Qe_k}{e_k}\int_{0}^{\infty}e^{-2\lambda_k s}\ud s 
\leq \frac{1}{2}\|Q\|_{\text{op}} \sum_{k=1}^{\infty}\frac{1}{\lambda_k}, 
	\end{align*} 
which verifies Condition \eqref{finite_trace_condition}  due to our assumption\eqref{exp.condition}.  We next show that \eqref{condition2_levy_measure_inv_iff} and \eqref{condition_levy_measure_inv_iff} are satisfied.
	It follows by Lemma 3.1 in \cite{OU} that for any $c >0$, 
	\begin{align}	\label{sup_levy_cond}
	K_c:= \sup_{\|u^*\|\leq c}\int_U \left(\langle u,u^*\rangle^2 \wedge 1\right) \, \mu(\udd u)< \infty.
	\end{align}
	For $k, m, n \in \N$, $m \leq n$ and $s \geq 0$, we define the following sets
	\begin{align*}
	C_k(s)&:=\left\{u:|\langle u,e_k\rangle| < \exp\left(\frac{\lambda_k s}{2}\right)\right\} \label{de.set_ck}\numberthis\\ 
	B_{m,n} &:=\left\{u:\sum_{k=m}^{n}\langle u,e_k\rangle^2 \leq 1\right\}.
	\end{align*}
	Using the spectral decomposition \eqref{heat_semigroup} of the semigroup, we have
	\begin{align*}
	\int_{0}^{\infty}\int_U & \left( \sum_{k=m}^n \langle u, T^*(s)e_k\rangle^2 \wedge 1 \right) \mu(\udd u)\ud s\\
	&= \int_{0}^{\infty}\int_U\left( \sum_{k=m}^n e^{-2\lambda_k s}\langle u, e_k\rangle^2 \wedge 1 \right) \mu(\udd u)\ud s\\
	& \leq \int_{0}^{\infty}\int_{B_{m,n}}\left( \sum_{k=m}^n e^{-2\lambda_k s}\langle u, e_k\rangle^2  \right)\, \mu(\udd u)\ud s\\
	& \qquad +  \int_{0}^{\infty}\int_{\cap_{j=m}^nC_j(s)\cap B_{m,n}^c}\left( \sum_{k=m}^n e^{-2\lambda_k s}\langle u, e_k\rangle^2  \right) \mu(\udd u)\ud s\\
	& \qquad +  \int_{0}^{\infty}\int_{\cup_{j=m}^nC^c_j(s)\cap 	B_{m,n}^c}\, \mu(\udd u)\ud s\\
	&=: I^1_{m,n}+I^2_{m,n}+I^3_{m,n}. \numberthis \label{ex.heat_maineq}
	\end{align*}
	Since $B_{m,n} \subseteq \cap_{k=m}^n\{u: \langle u,e_k\rangle^2 \leq 1\}$, we have
	\begin{align*}
	I^1_{m,n}&\leq \sum_{k=m}^n\left(\int_{0}^{\infty}e^{-2\lambda_k s}\ud s\right) \int_{\{u:\langle u,e_k\rangle^2 \leq 1\}}  \langle u, e_k\rangle^2 \, \mu(\udd u)\\
	& \leq \sum_{k=m}^n\left(\frac{1}{2\lambda_k}\right)\sup_{\|u^*\|\leq 1}\int_U \left(\langle u, u^*\rangle^2 \wedge 1\right) \, \mu(\udd u)
	 = \frac{K_1}{2} \sum_{k=m}^n\frac{1}{\lambda_k}. \numberthis \label{ex.heat_term1}
	\end{align*}
	Using the definition of the sets $C_k(s)$, we obtain
	\begin{align*}
	I^2_{m,n} &\leq  \int_{0}^{\infty}\int_{\cap_{j=m}^nC_j(s)} \sum_{k=m}^n e^{-\lambda_k s}\left(e^{-\lambda_k s}\langle u, e_k\rangle^2 \wedge 1 \right) \mu(\udd u)\ud s\\
	&\leq \sum_{k=m}^n\left(\int_{0}^{\infty}e^{-\lambda_k s}\ud s\right) \int_U \left(\langle u,e_k\rangle^2 \wedge 1\right) \, \mu(\udd u)\\
	& \leq \sum_{k=m}^n\frac{1}{\lambda_k}\sup_{\|u^*\|\leq 1}\int_U \left(\langle u,u^*\rangle^2 \wedge 1\right) \, \mu(\udd u)
	 = K_1 \sum_{k=m}^n\frac{1}{\lambda_k} \numberthis \label{ex.heat_term2}
	\end{align*}
	Noting that $\cup_{k=m}^{n}C_k^c(s) = \left\{u: \max_{m\leq k\leq n}\left(\frac{2\log^+|\langle u,e_k\rangle|}{\lambda_k} \right) > s\right\}$, we have
	\begin{align*}
	I^3_{m,n} &\leq \int_{0}^{\infty}\int_{\cup_{k=m}^nC^c_k(s)}\, \mu(\udd u)\ud s\\
& = \int_0^{\infty}\mu \left(\left\{u: \max_{m\leq k\leq n}\left(\frac{2\log^+|\langle u,e_k\rangle|}{\lambda_k} \right) > s\right\}\right)\ud s\\
	& = 2\int_U \max_{m\leq k\leq n}\frac{\log^+|\langle u,e_k\rangle|}{\lambda_k}\, \mu(\udd u), \numberthis \label{ex.heat_term3}
	\end{align*}
	where the last equality follows from a Fubini argument;
	see \cite[App.\ 2]{Billingsley} for details. 
	Hence substituting \eqref{ex.heat_term1}, \eqref{ex.heat_term2} and \eqref{ex.heat_term3} in \eqref{ex.heat_maineq} and using \eqref{exp.condition} and \eqref{logcond_compact_semigroup} verifies Condition \eqref{condition_levy_measure_inv_iff}. Similarly, by setting $m=1$ in  \eqref{ex.heat_term1}, \eqref{ex.heat_term2} and \eqref{ex.heat_term3} and  using \eqref{exp.condition} and \eqref{logcond_sup_compact_semigroup} verifies Condition 
	 \eqref{condition2_levy_measure_inv_iff}. 

	
It remains to  prove that \eqref{condition_cinfty_exist} is satisfied, that is there exists $\lim_{t\to \infty} c_t$ where for each $v \in V$,
	\begin{align} \label{cond_1}
	\langle c_t,v\rangle:=\int_0^t a(T^*(s)v)\ud s +\int_V\langle h, v\rangle \big(\1_{B_V}(h)-\1_{B_{\mathbb{R}}}(\langle h, v\rangle)\big)\,\xi_t(\udd h).
	\end{align}
	We first prove that $\int_0^{\infty}|a(T^*(s)v)|\ud s < \infty$ for each $v \in V$. For this we will  use the following equality which holds for all $u^* \in U$ and $\beta >0$ as given by (3.9) in \cite{OU}, 
	\begin{align} \label{decomp_drift_a}
	a(\beta u^*)=\beta a(u^*)+\beta \int_U\langle u,u^*\rangle \big(\1_{B_{\mathbb{R}}}(\beta \langle u,u^*\rangle)-\1_{B_{\mathbb{R}}} ( \langle u,u^*\rangle)\big)\,\mu(\udd u).
	\end{align}
	Let $\pi_{n}:U \to U$ be the projection operator defined by $\pi_{n}(v) := \sum_{k=1}^n\langle v,e_k\rangle e_k$. Then by \eqref{decomp_drift_a} and
	assuming $\|T^*(s)\pi_nv\|\neq 0$ , we obtain
	\begin{align*}
	&|a(T^*(s)\pi_nv)|\\
	&\leq\|T^*(s)\pi_nv\|\left|a\left(\frac{T^*(s)\pi_nv}{\|T^*(s)\pi_nv\|}\right)\right|\numberthis \label{secondterm_in_a} \\
	&\quad + \int_U\left|\left\langle u,T^*(s)\pi_nv \right\rangle\right|\left|\1_{B_{\mathbb{R}}}(\langle u,T^*(s)\pi_nv\rangle)-\1_{B_{\mathbb{R}}} \left( \left \langle u,\frac{T^*(s)\pi_nv}{\|T^*(s)\pi_nv\|}\right\rangle\right)\right|\,\mu(\udd u). 
	\end{align*}
	Since $a$ maps bounded sets to bounded sets, it follows by \eqref{eq.heat_semigp_exp_stable} that
	\begin{align*}
	\int_0^{\infty}\|T^*(s)\pi_nv\|\left|a\left(\frac{T^*(s)\pi_nv}{\|T^*(s)\pi_nv\|}\right)\right|\ud s &\leq \|v\|\sup_{\|u^*\|\leq 1}|a(u^*)|\int_0^{\infty}e^{-\lambda_1 s}\ud s\\
	& =  \frac{\|v\|}{\lambda_1}\sup_{\|u^*\|\leq 1}|a(u^*)|. \numberthis
	\label{estimate_a}
	\end{align*}
Integrating the second term on the right side in \eqref{secondterm_in_a} results in 
	\begin{align*}
	\int_0^{\infty}\int_U&\left|\left\langle u,T^*(s)\pi_nv\right \rangle\right|\left|\1_{B_{\mathbb{R}}}(\langle u,T^*(s)\pi_nv\rangle)-\1_{B_{\mathbb{R}}} \left( \left \langle u,\frac{T^*(s)\pi_nv}{\|T^*(s)\pi_nv\|}\right\rangle\right)\right|\,\mu(\udd u)\ud s\\
	&=\int_0^{\infty}\int_{\left\{u:\left|\langle u,T^*(s)\pi_nv\rangle\right|> 1\right\}\cap \left\{u:\left|\left \langle u,\frac{T^*(s)\pi_nv}{\|T^*(s)\pi_nv\|}\right\rangle\right|\leq 1\right\}}\left|\left\langle u,T^*(s)\pi_nv\right\rangle\right|\,\mu(\udd u)\ud s\\
	&\qquad + \int_0^{\infty}\int_{\left\{u:\left|\langle u,T^*(s)\pi_nv\rangle\right|\leq 1\right\}\cap \left\{u:\left|\left \langle u,\frac{T^*(s)\pi_nv}{\|T^*(s)\pi_nv\|}\right\rangle\right|>1\right\}}\left|\left\langle u,T^*(s)\pi_nv\right\rangle\right|\, \mu(\udd u)\ud s\\
	& =: I_4 +I_5. \numberthis
	\end{align*}
	Using \eqref{eq.heat_semigp_exp_stable} we estimate $I_4$ by
	\begin{align*}
I_4 &\leq \int_0^{\infty}\int_{\left\{u:\left|\left \langle u,\frac{T^*(s)\pi_nv}{\|T^*(s)\pi_nv\|}\right\rangle\right|\leq 1\right\}}\left\langle u,T^*(s)\pi_nv\right\rangle^2\mu(\udd u)\ud s\\
	&\leq  \int_0^{\infty}\|T^*(s)\pi_nv\|^2\int_{U}\left(\left \langle u,\frac{T^*(s)\pi_nv}{\|T^*(s)\pi_nv\|}\right\rangle^2 \wedge 1\right)\,  \mu(\udd u)\ud s\\
	&\leq \|v\|^2 \sup_{\|u^*\|\leq 1}\int_U\big(\langle u, u^*\rangle^2 \wedge 1\big)\, \mu (\udd u) \int_0^{\infty}e^{-2\lambda_1 s} \ud s
	= \frac{\|v\|^2}{2\lambda_1} K_1.  \numberthis
	\end{align*}
	If  $C_k(s)$ denotes the set defined in \eqref{de.set_ck}, we obtain
	\begin{align*}
	I_5 & \leq \int_{0}^{\infty}\int_{\left\{u:\left|\left \langle u,\frac{T^*(s)\pi_nv}{\|T^*(s)\pi_nv\|}\right\rangle\right|>1\right\}\cap (C_1(s) \cap \cdots \cap C_n(s))}\left|\left\langle u,T^*(s)\pi_{n}v\right\rangle\right|\,\mu(\udd u)\ud s\\
	&  \qquad + 	\int_{0}^{\infty}\int_{\left\{u:\left|\left \langle u,\frac{T^*(s)\pi_nv}{\|T^*(s)\pi_nv\|}\right\rangle\right|>1\right\}\cap (C^c_1(s) \cup \cdots \cup C^c_n(s))}\mu(\udd u)\ud s\\
	& \leq 	\int_{0}^{\infty}\int_{\left\{u:\left|\left \langle u,\frac{T^*(s)\pi_nv}{\|T^*(s)\pi_nv\|}\right\rangle\right|>1\right\}\cap (C_1(s) \cap \cdots \cap C_n(s))}\sum_{k=1}^n|\langle u,e_k\rangle||\langle v,e_k\rangle|e^{-\lambda_ks}\,\mu(\udd u)\ud s\\
	& \qquad \qquad + 	\int_{0}^{\infty}\mu\big(C^c_1(s) \cup \cdots \cup C^c_n(s)\big)\ud s\\
	& =: I_6+I_7. \numberthis
	\end{align*}
For the integral $I_6$ we obtain
	\begin{align*}
	I_6 &\leq \sum_{k=1}^n|\langle v,e_k\rangle|\int_{0}^{\infty}e^{-\frac{\lambda_ks}{2}}\int_{\left\{u:\left|\left \langle u,\frac{T^*(s)v}{\|T^*(s)v\|}\right\rangle\right|>1\right\}}\mu(\udd u)\ud s\\
	& \leq 2\|v\|\sup_{\|u^*\|\leq 1}\int_U\left(\langle u, u^*\rangle^2 \wedge 1\right)\, \mu (\udd u) \sum_{k=1}^n\frac{1}{\lambda_k}\\
	& \leq 2\|v\|K_1 \sum_{k=1}^{\infty}\frac{1}{\lambda_k}, \numberthis \label{estimate_I6}
	\end{align*}
	which is finite by using \eqref{exp.condition}.
Using the same equality from \cite[App.\ 2]{Billingsley} as in \eqref{ex.heat_term3}, we obtain 
	\begin{align*}
	I_7 & \leq \int_{0}^{\infty}\mu\left(\cup_{k=1}^{n}\left\{u:|\langle u,e_k\rangle| \ge e^{\frac{\lambda_k s}{2}}\right\}\right)\ud s\\
	& = \int_{0}^{\infty}\mu\left(\left\{u:\max_{1\leq k\leq n}\frac{2}{\lambda_k}\log^+|\langle u,e_k\rangle| \ge s\right\}\right)\ud s\\
	& =2\int_U\max_{1\leq k\leq n}\frac{1}{\lambda_k}\log^+|\langle u,e_k\rangle|\,\mu(\udd u).  \numberthis \label{estimate_I7}
	\end{align*}
Applying \eqref{logcond_sup_compact_semigroup} to \eqref{estimate_I7}
and using \eqref{estimate_a} -- \eqref{estimate_I7}, it follows from 
\eqref{secondterm_in_a} that  there exists some $C_1>0$ such that
	\[\sup_{n\in \N}\int_0^{\infty}|a(T^*(s)\pi_nv)|\ud s \leq C_1(\|v\|+1).\]
	Fatou's Lemma implies that for any $\delta >0$,
	\begin{align*}
	M_{\delta} &:= \sup_{\|v\|<\delta}\int_0^{\infty}|a(T^*(s)v)|\ud s \\
	& \leq \sup_{\|v\|<\delta} \liminf_{n\to\infty}\int_0^{\infty}|a(T^*(s)\pi_n v)|\ud s
	 \leq \sup_{\|v\|<\delta}  C_1(\|v\|+1)
	< \infty. \label{eq.Mdelta} \numberthis
	\end{align*}
This proves that $\int_0^{\infty}a(T^*(s)v)\ud s$ exists, and,  for each $v \in V$, we have
\begin{align}\label{eq.weak-c-t-part1}
\lim_{t\to\infty} \int_0^t a(T^\ast(s)v)\,ds =\int_0^\infty 
a(T^\ast(s)v)\,ds. 
\end{align}
For considering the second term in \eqref{cond_1}, define 
 $f(h,v):= \langle h, v\rangle \left(\1_{B_V}(h)-\1_{B_{\mathbb{R}}}(\langle h, v\rangle)\right)$. Then for any $h, v \in V$,
 \begin{align*}
 |f(h,v)| &=  |\langle h, v\rangle|\1_{B_V}(h)\1_{B^c_{\mathbb{R}}}(\langle h, v\rangle)+|\langle h, v\rangle|\1_{B^c_V}(h)\1_{B_{\mathbb{R}}}(\langle h, v\rangle)\\
 &\le |\langle h, v\rangle|^2\1_{B_V}(h)+\1_{B^c_V}(h),
 \end{align*} 
 from which the  integrability of $f(\cdot, v)$ with respect to $\xi_\infty$ follows by using the properties of L\'evy measures.
Consequently, since $\xi_t(C) \uparrow \xi_{\infty}(C)$ as $ t \to \infty$ for each $C \in \mathcal{B}(V)$
 and $\xi_\infty$ is a L\'evy measure due to Lemma \ref{le.levy_ext_eqv}, we obtain by the same arguments as in Lemma 3.3 in \cite{rockner_fuhrman} that
	\[\lim_{t\rightarrow \infty}\int_Vf(h,v)\,\xi_t(\udd h) =\int_Vf(h,v)\,\xi_{\infty}(\udd h).\]
Together with \eqref{eq.weak-c-t-part1} it follows from \eqref{cond_1} that $(\scapro{c_t}{v})_{t\ge 0}$ converges for each $v\in V$ and 
	\[\lim_{t \to \infty}\scapro{c_t}{v}=\int_0^{\infty}a(T^*(s)v)\ud s +\int_Vf(h,v)\xi_{\infty}(\udd h). \] 
To prove that $(c_t)_{t\ge 0}$ converges in $V$, it is enough to show that $(c_t)_{t\ge 0}$ is relatively compact in $V$, which in this case, as $(c_t)_{t\ge 0}$ is bounded, reduces to establish
	\begin{align}
	\lim_{m\to \infty}\sup_{t\geq 0} \sum_{k=m}^{\infty}\scapro{c_t}{e_k}^2 = 0.
	\end{align}
	Using \eqref{eq.Mdelta}, Cauchy-Schwarz inequality and the fact that $\xi_t \le \xi_{\infty}$, we obtain 
	\begin{align*}
&	\scapro{c_t}{e_k}^2 \\
&= \left(\int_0^t a(T^*(s)e_k)\ud s +\int_V\langle h, e_k\rangle \left(1_{B_V}(h)-1_{B_{\mathbb{R}}}(\langle h, e_k\rangle)\right)\,\xi_t(\udd h)\right)^2\\
	& \leq 2 \left(\int_0^t |a(T^*(s)e_k)|\ud s\right)^2 +2\left(\int_{|\scapro{h}{e_k}|\leq 1 < \|h\|}\langle h, e_k\rangle \,\xi_t(\udd h)\right)^2\\
	&\leq 2 \left(\sup_{\|v\|\leq 1}\int_0^\infty |a(T^*(s)v)|\ud s\right) \int_0^\infty |a(T^*(s)e_k)|\ud s\\
	&\qquad  +2 \xi_t(\|h\|>1)\int_{|\scapro{h}{e_k}|\leq 1 < \|h\|}\langle h, e_k\rangle^2 \,\xi_t(\udd h)\\
	&\leq 2 M_1 \int_0^\infty |a(T^*(s)e_k)|\ud s +2 \xi_{\infty}(\{h:\|h\|>1\})\int_{|\scapro{h}{e_k}|\leq 1 }\langle h, e_k\rangle^2 \,\xi_{\infty}(\udd h). \numberthis \label{ct_compact_term_ek}
	\end{align*}
	It follows by \eqref{decomp_drift_a} and Fubini's theorem that
	\begin{align*}
	\int_0^\infty &|a(T^*(s)e_k)|\ud s\\
	&\le |a(e_k)|\int_0^\infty e^{-\lambda_ks}\ud s + \int_0^\infty e^{-\lambda_k s}\int_{1 \le |\langle u,e_k\rangle| \le e^{\lambda_ks}}|\langle u,e_k\rangle|\, \mu(\udd u)\ud s\\
	& =\frac{|a(e_k)|}{\lambda_k} + \int_0^\infty e^{-\lambda_k s}\int_{1 \le |\beta| \le e^{\lambda_ks}}|\beta|\, \left(\mu \circ \langle \cdot, e_k\rangle^{-1}\right)(\udd \beta)\ud s\\
	&  =\frac{|a(e_k)|}{\lambda_k} + \int_{1 \le |\beta|}|\beta|\int_{\frac{1}{\lambda_k}\log |\beta|}^\infty e^{-\lambda_k s}\ud s \, \left(\mu \circ \langle \cdot, e_k\rangle^{-1}\right)(\udd \beta)\\
	&  =\frac{|a(e_k)|}{\lambda_k} + \frac{1}{\lambda_k}\int_{1 \le |\beta|}\left(\mu \circ \langle \cdot, e_k\rangle^{-1}\right)(\udd \beta)\\
	&  =\frac{|a(e_k)|}{\lambda_k} + \frac{1}{\lambda_k}\mu\left(\{u:|\scapro{u}{e_k}|\ge 1\}\right)\\
	& \leq \frac{1}{\lambda_k} \left(\sup_{ \|u^*\|\le 1}|a(u^*)|+ K_1\right). \numberthis \label{ct_compact_term_a}
	\end{align*}
Similar application of  Fubini's theorem implies for the second term in 
\eqref{ct_compact_term_ek} that
\begin{align*}
&\int_{|\scapro{h}{e_k}|\leq 1}\langle h, e_k\rangle^2 \xi_{\infty}(\udd h)\\
	& =  \int_0^\infty e^{-2\lambda_k s}\int_{|\langle u,e_k\rangle| \le e^{\lambda_ks}}|\langle u,e_k\rangle|^2\,\mu(\udd u)\ud s\\
	& =  \int_0^\infty e^{-2\lambda_k s}\int_{|\langle u,e_k\rangle| \le 1}|\langle u,e_k\rangle|^2\,\mu(\udd u)\ud s\\
&\qquad\qquad 	 +  \int_0^\infty e^{-2\lambda_k s}\int_{1 <|\langle u,e_k\rangle| \le e^{\lambda_ks}}|\langle u,e_k\rangle|^2\,\mu(\udd u)\ud s\\
	& \le  \frac{1}{2\lambda_k}\int_{U}\left(|\langle u,e_k\rangle|^2\wedge 1\right) \mu(\udd u) +  \int_0^\infty e^{-2\lambda_k s}\int_{1 <|\beta| \le e^{\lambda_ks}}|\beta|^2\,\left(\mu \circ \langle \cdot, e_k\rangle^{-1}\right)(\udd \beta)\ud s\\
	& =  \frac{1}{2\lambda_k}K_1+  \int_{1 \le |\beta|}|\beta|^2\int_{\frac{1}{\lambda_k}\log |\beta|}^\infty e^{-2\lambda_k s}\ud s \left(\mu \circ \langle \cdot, e_k\rangle^{-1}\right)(\udd \beta)\\
	& \le  \frac{1}{\lambda_k}K_1.  \numberthis \label{ct_compact_term_xi}
	\end{align*}
	Using the estimates obtained in \eqref{ct_compact_term_a} and \eqref{ct_compact_term_xi} in \eqref{ct_compact_term_ek}, it follows that
there exists a constant $C_2>0$ such that 
	\begin{align*}
	\sup_{t\geq 0}\sum_{k=m}^\infty\scapro{c_t}{e_k}^2 & \leq C_2\sum_{k=m}^\infty\frac{1}{\lambda_k},
	\end{align*}
	which implies that $(c_t)_{t\ge 0}$ is relatively compact in $V$ and hence \eqref{condition_cinfty_exist} is satisfied.
	
	(a) $\Rightarrow$ (b). 
Using the same equality from \cite[App.\ 2]{Billingsley} as in \eqref{ex.heat_term3},	 we have for $m, n \in \N$ with $m \le n$ that
	\begin{align*}
	\int_U \max_{m\leq k\leq n}\frac{\log^+|\langle u,e_k\rangle|}{\lambda_k}\, \mu(\udd u) 
	& = \int_0^{\infty}\mu \left(\left\{u: \max_{m\leq k\leq n}\left(\frac{\log^+|\langle u,e_k\rangle|}{\lambda_k} \right) > s\right\}\right) \ud s\\
	&=\int_0^{\infty} \mu \left(\cup_{k=m}^n\left\{u:e^{- \lambda_ks}|\langle u,e_k\rangle| > 1\right\} \right) \ud s\\
	&\le\int_0^{\infty} \mu \left(\sum_{k=m}^n e^{- 2\lambda_ks}|\langle u,e_k\rangle|^2 > 1 \right) \ud s\\
	& \leq \int_0^{\infty}\int_U\left( \sum_{k=m}^n \langle u, T^*(s)e_k\rangle^2 \wedge 1 \right) \mu(\udd u)\ud s.
	\end{align*}
Since combining Theorem \ref{th.stable_case} with Theorem \ref{condiff} implies 
Conditions \eqref{condition2_levy_measure_inv_iff}	 and \eqref{condition_levy_measure_inv_iff}, the above inequality verifies 
Conditions \eqref{logcond_sup_compact_semigroup} and \eqref{logcond_compact_semigroup}, 
which completes the proof.
\end{proof}
\begin{remark}\label{re.extra_cond_inv.meas}
	From the proof of Theorem \ref{th.heat_semigrp_iff} it also follows that without assuming \eqref{exp.condition}, the following conditions:
	\begin{enumerate}
		\item [(iii)] \itemEq{ \sum_{k=1}^{\infty}\frac{|a(e_k)|}{\lambda_k}< \infty;}
		\item [(iv)] \itemEq{ \sum_{k=1}^{\infty}\frac{\scapro{Qe_k}{e_k}}{\lambda_k}< \infty;}
		\item[(v)] \itemEq{\sum_{k=1}^{\infty}\frac{1}{\lambda_k}\int_U \left( \scapro{u}{e_k}^2 \wedge 1\right) \mu (\udd u) < \infty, \label{eq.levy-measure_spectral.cond}}
	\end{enumerate} together with  \eqref{logcond_sup_compact_semigroup} and \eqref{logcond_compact_semigroup} are  sufficient for the existence of an invariant measure.
\end{remark}

\begin{remark}
Let $L$ be a genuine L\'evy process  with classical characteristics $(b, Q, \mu)$. According to Theorem 1 in \cite{michalik_heat}, if the semigroup satisfies \eqref{eq.conditon-eigenvalues}, \eqref{heat_semigroup} and 
	\begin{align}\label{exp.condition.michalik}
	\sum_{k=1}^{\infty}\frac{e^{-\lambda_kT}}{\lambda_k} < \infty,
	\end{align}
then a necessary and sufficient condition for the existence of a stationary measure for the process \eqref{OUP} is given by
	\begin{align}
	\int_U \sup_{ n\in\N} \left(\frac{\log^+|\langle u,e_n\rangle|}{\lambda_n} \right) \mu(\udd u) < \infty. \label{logcond_classical_levy}
	\end{align}
	In this case, our Conditions  \eqref{logcond_sup_compact_semigroup} and \eqref{logcond_compact_semigroup} are  equivalent to \eqref{logcond_classical_levy}.  Since $\mu$ is a genuine L\'evy measure, the monotone convergence theorem implies
	\[\sup_{n\geq m} \int_U \max_{m\leq k\leq n} \left(\frac{\log^+|\langle u,e_k\rangle|}{\lambda_k} \right) \mu(\udd u) = \int_U \sup_{n\geq m} \left(\frac{\log^+|\langle u,e_n\rangle|}{\lambda_n} \right) \mu(\udd u),\] 
	which shows the equivalence of \eqref{logcond_sup_compact_semigroup} and \eqref{logcond_classical_levy} by taking $m=1$.
Furthermore, Condition \eqref{eq.conditon-eigenvalues} implies for each $u \in U$ that
	\begin{align*}
	\sup_{n\geq m} \frac{\log^+|\langle u,e_n\rangle|}{\lambda_n} \leq \sup_{n\geq m}\frac{\log^+\| u\|}{\lambda_n} \leq \frac{\log^+\| u\|}{\lambda_m} \to 0 \text{    as  } m \to \infty.
	\end{align*}
	An application of Lebesgue's theorem together with \eqref{logcond_classical_levy}  implies \eqref{logcond_compact_semigroup}. 
	
Condition \eqref{exp.condition.michalik} is weaker than \eqref{exp.condition}. But it is well known (and also  mentioned in above Example) that the stochastic heat equation driven by a cylindrical Brownian motion has a weak solution if and only if $d =1$. Therefore, condition \eqref{exp.condition} is more natural for an arbitrary cylindrical L\'evy process. But if $L$ is a genuine L\'evy process or if $L$ has characteristics $(0,0,\mu)$ where $\mu$ is symmetric, then the above proof can be easily modified by assuming \eqref{exp.condition.michalik} instead of \eqref{exp.condition}.
\end{remark}

\begin{example}\label{ex.series-invariant}
The following specific example of a cylindrical L\'evy process is often considered in the literature:
let $(\ell_k)_{k\in \N}$ be a sequence of symmetric, independent, real valued L\'evy  processes with characteristics $(0,0,\mu_k)$, and define
\begin{align}\label{eq.example-L-series}
 L(t)u:=\sum \limits_{k=1}^{\infty}\langle e_k, u\rangle  \ell_k(t)
	\qquad\text{for all }u\in U,\, t\ge 0.
\end{align}
If the sum converges for each $u\in U$ and the family of characteristic functions of $\ell_k$ are equicontinuous in $0$, then \eqref{eq.example-L-series} defines a cylindrical L\'evy process $L$; see \cite[Le.\ 4.2]{OU}. 

Assume that the semigroup $(T(t))_{t\ge 0}$ satisfies  the spectral representation \eqref{heat_semigroup}. Since the independence of the real valued processes $(\ell_k)_{k\in\N}$ implies that the cylindrical L\'evy measure $\mu$  is concentrated on the axes, Conditions \eqref{logcond_sup_compact_semigroup} and  \eqref{logcond_compact_semigroup} reduce to
	\begin{align}\label{eq.expstable-stationary-series}
	\sum_{k=1}^{\infty}\frac{1}{\lambda_k}\int_{\R}\log^+|\beta|\,\mu_k(\udd \beta) < \infty.
	\end{align}
\end{example}

\begin{example} \label{ex.priola}
The authors of \cite{priola_zabczyk_cyl} consider  a cylindrical L\'evy process of the form \eqref{eq.example-L-series} with $\ell_k=\sigma_k m_k$ for all $k\in\N$, where 
$(\sigma_k)_{k\in\N}\subseteq \ell^\infty(\R)$ and $(m_k)_{k\in\N}$ is a sequence of identically distributed, independent, symmetric real valued L\'evy processes without Gaussian part and with L\'evy measure $\rho$.
In this case, Condition \eqref{eq.expstable-stationary-series} is equivalent to 
\begin{align}\label{eq.log-condition-series}
	\sum_{k=1}^{\infty}\frac{1}{\lambda_k}\int_{\R}\log^+|\sigma_k\beta|\,\rho(\udd \beta) < \infty.
\end{align}
It follows from Theorem \ref{th.heat_semigrp_iff} that, if the reciprocal eigenvalues $(1/\lambda_k)_{k\in\N}$ are summable, i.e.\ satisfy \eqref{exp.condition}, then there exists a stationary solution if and only if
\begin{align} \label{log_cond_rho}
\int_1^{\infty}\log \beta\, \rho(\udd \beta) <\infty.
\end{align}
This covers exactly the result in \cite{priola_zabczyk_cyl}. 

However, Theorem \ref{condiff} improves the results from \cite{priola_zabczyk_cyl}: without assuming that  the reciprocal eigenvalues $(1/\lambda_k)_{k\in\N}$ are summable, there exists a stationary measure if and only if Conditions \eqref{condition2_levy_measure_inv_iff} and \eqref{condition_levy_measure_inv_iff}	are satisfied which in this case is equivalent to
\begin{align}\label{eq.iff-condition-series}
\sum_{k=1}^\infty \int_0^\infty \int_{\R} \left(e^{-2 \lambda_k s}|\sigma_k \beta|^2 
\wedge 1\right)\, \rho(\udd \beta)\, \udd s <\infty.
\end{align}
Furthermore, Condition \eqref{eq.iff-condition-series} is satisfied if and only if Condition \eqref{eq.log-condition-series} is true and
\begin{align}
\sum_{k=1}^\infty	\frac{1}{\lambda_k}\int_{\R} \left(|\sigma_k \beta|^2 \wedge 1 \right)\rho(\udd \beta) &< \infty. \label{eq.iff-series.condition1}
\end{align}
The last claim follows from the following calculation: 
\begin{align*}
\int_0^\infty &\int_{\R} \left(e^{-2 \lambda_k s}|\sigma_k \beta|^2 
\wedge 1\right)\, \rho(\udd \beta)\, \udd s\\ 
& = \int_{\R} \int_0^\infty \left(e^{-2 \lambda_k s}|\sigma_k \beta|^2 
\wedge 1\right)\, \udd s \, \rho(\udd \beta)\\
& = \int_{|\sigma_k\beta| \le 1} \int_0^\infty e^{-2 \lambda_k s}|\sigma_k \beta|^2 
\, \udd s \, \rho(\udd \beta) +\int_{|\sigma_k\beta| > 1} \int_0^\infty \left(e^{-2 \lambda_k s}|\sigma_k \beta|^2 
\wedge 1\right)\, \udd s \, \rho(\udd \beta)\\
& = \frac{1}{2\lambda_k}\int_{|\sigma_k\beta| \le 1} |\sigma_k \beta|^2 \rho(\udd \beta)+\frac{1}{\lambda_k}\int_{|\sigma_k\beta| > 1} \log |\sigma_k\beta| \, \rho(\udd \beta) +\frac{1}{2\lambda_k}\int_{|\sigma_k\beta| > 1} \, \rho(\udd \beta)\\
& = \frac{1}{2\lambda_k}\int_{\R} \left(|\sigma_k \beta|^2 \wedge 1 \right)\rho(\udd \beta) +\frac{1}{\lambda_k}\int_{\R} \log^+|\sigma_k\beta| \, \rho(\udd \beta).
\end{align*}

For example, if each $m_k$ is chosen as a symmetric, $\alpha$-stable process with L\'evy measure $\rho(\udd \beta) = \frac{1}{2}|\beta|^{-1-\alpha} \udd \beta$, then a simple calculation shows that \eqref{eq.iff-condition-series} is satisfied if and only if 
\begin{align*}
\sum_{k=1}^\infty \frac{|\sigma_k|^\alpha}{\lambda_k}<\infty,
\end{align*}
which is a weaker condition than assuming 
that the reciprocal eigenvalues are summable. 
\end{example}

\bibliographystyle{plain}

\end{document}